\documentclass{amsart}
\usepackage{geometry}                
\usepackage{graphicx}
\usepackage{amssymb}
\usepackage{epstopdf,amsmath}
\usepackage{enumitem}

\usepackage{tikz}
\usepackage{geometry}
\usepackage{pgfplots}

\usepgfplotslibrary{colorbrewer}
\pgfplotsset{width=8cm,compat=1.9}

\newtheorem{prop}{Proposition}[section]

\newtheorem{lemma}[prop]{Lemma}

\newtheorem{theorem}[prop]{Theorem}

\newtheorem{defn}[prop]{Definition}
\newtheorem*{prop*}{Proposition}
\newtheorem*{theorem*}{Theorem}

\newtheorem{mlem}{Main Lemma}
\newtheorem*{mlem*}{Main Lemma}
\newtheorem*{lemma*}{Lemma}

\newtheorem{ques}[prop]{Question}

\newcommand{\VV}{\mathbb{V}}
\newcommand{\WW}{\mathbb{W}}
\newcommand{\PP}{\mathbb{P}}

\newcommand{\TT}{\mathbb{T}}
\newcommand{\SSS}{\mathbb{S}}

\newcommand{\eps}{\epsilon}

\newcommand{\RR}{\mathbb{R}}

\newcommand{\XX}{\mathbb{X}}

\newcommand{\exmult}{\eta_{\textrm{mult}}}
\newcommand{\exscal}{\epsilon_{\textrm{scal}}}

\newcommand{\exsticky}{\epsilon_{\textrm{sticky}}}
\newcommand{\exfact}{\eta_{\textrm{factor}}}

\title{Outline of the Wang-Zahl proof of the Kakeya conjecture in $\RR^3$}
\author{Larry Guth}

\begin{document}

\maketitle

In the recent preprint \cite{WZ},  Hong Wang and Joshua Zahl proved the 3-dimensional Kakeya conjecture, building on previous papers \cite{WZ1} and \cite{WZ2}.   I am working on digesting and checking the proof.  I wrote an expository article \cite{G} giving an overview of the Kakeya problem and the proof as a whole.  

The papers \cite{WZ1} and \cite{WZ2} prove the sticky case of the Kakeya conjecture, roughly following an approach suggested by Katz and Tao, and dealing with significant technical issues.  The paper \cite{WZ} reduces the general case to the sticky case.  Both parts of the proof are difficult and significant, but reducing to the sticky case is the part that seemed completely out of reach until recently.  

In these followup notes, we give a detailed outline of how to reduce the general Kakeya conjecture to the sticky case.  The argument here is a small variation on the one in \cite{WZ}, which I hope simplifies some technical details.

\section{Introduction to the Kakeya problem}

Suppose that $\TT$ is a set of tubes in $\RR^3$ with radius $\delta$ and length 1.   We write

$$U(\TT) = \bigcup_{T \in \TT} T. $$

Our goal is to understand $|U(\TT)|$, the volume of $U(\TT)$.   We write $|\TT|$ for the cardinality of $\TT$ and we write $|T|$ for the volume of a tube $T \in \TT$.  We begin with a naive question.

\begin{ques} If $\TT$ is a set of $\delta$-tubes in $\RR^3$,  is it true that

$$|U(\TT)| \gtrapprox | \TT | |T| ? $$

\end{ques}

The answer to this question is no.  One problem is that all of the tubes of $\TT$ could be almost equal to each other,  so that $|U(\TT)|$ is almost $|T|$.  To avoid this,  we assume that the tubes of $\TT$ are essentially distinct, meaning that $|T_1 \cap T_2| \ge (1/2) |T_1|$ for any distinct $T_1,  T_2 \in \TT$.  Throughout the survey,  we will assume that we are working with essentially distinct sets of tubes.

But even if we assume that the tubes of $\TT$ are essentially distinct, the answer to this question is still no.  Here is a key counterexample.   Suppose that $K \subset \RR^3$ is a convex set with dimensions $a \times b \times 1$,  with $\delta \le a,b$.   Suppose that $\TT$ is a maximal set of essentially distinct tubes contained in $K$.   If $|K|$ is much bigger than $|T|$,  then we will have $|\TT| |T| \gg |K|$.   

In this example,  the tubes $\TT$ cluster in the convex set $K$.   Next we would like to consider a set of tubes that does not cluster in any convex sets.   First we need some language to make this precise.

Suppose that $\WW$ is a set of convex sets in $\RR^n$.  
We will consider how the sets of $\WW$ cluster in various other convex sets.  More generally, suppose that $\WW$ is a finite set of convex sets in $\RR^n$ and suppose that $K$ is another convex set.  We define

\begin{equation} \label{defWK} \WW[K] := \{ W \in \WW : W \subset K \}
\end{equation}

\begin{equation} \label{defDelta} \Delta(\WW, K) := \frac{ \sum_{W \in \WW[K]} |W| } {|K|}
\end{equation}

We can think of $\Delta (\WW, K)$ as a kind of density that measures how $\WW$ concentrates in $K$.  We define a maximal density

\begin{equation} \label{defDeltamax} \Delta_{max}(\WW) := \max_{K \textrm{ convex}} \Delta(\WW, K)
\end{equation}

The inequality $\Delta_{max}(\WW) \lessapprox 1$ captures the idea that $\WW$ does not cluster in convex sets.

\begin{theorem} \label{thmkak1} (Kakeya conjecuture, Wang-Zahl, \cite{WZ}) Suppose that $\TT$ is a set of $\delta \times \delta \times 1$-tubes in $B_1 \subset \RR^3$ with $ \Delta_{max}(\TT) \lessapprox 1$.  Then 

$$ |U(\TT)| \gtrapprox | \TT| |T|. $$

\end{theorem}

Remark.  The original Kakeya conjecture concerns direction-separated tubes.   One version says that  if $\TT$ is a set of $\delta^{-(n-1)}$ $\delta$-tubes in $\RR^n$ with $\delta$-separated directions,  then $|U(\TT)| \gtrapprox 1$.   It is straightforward to check that if the tubes of $\TT$ are direction separated, then $\Delta_{max}(\TT) \lesssim 1$,  and so Theorem \ref{thmkak1} implies this form of the Kakeya conjecture in $\RR^3$. 

The 2-dimensional version of the Kakeya conjecture follows from a short argument based on Cauchy-Schwarz, which has been known for many years.  The 3-dimensional version of the Kakeya conjecture has been intensively investigated for many years.

It is worth mentioning that the direct analogue of Theorem \ref{thmkak1} is false in dimensions $n \ge 4$,  because of counterexamples where the tubes $\TT$ cluster in the $\delta$-neighborhood of a low degree algebraic variety.   For instance,  in $\RR^4$,  the tubes can cluster in the $\delta$-neighborhood of the quadratic hypersurface defined by $x_1^2 + x_2^2 - x_3^2 - x_4^2 = 1$.   It is possible that Theorem \ref{thmkak1} can be generalized to higher dimensions by replacing convex sets by semi-algebraic sets of small complexity. 

Instead of considering only the union of the tubes,  it is also natural to consider a subset of each tube,  which is called a shading.  
In general,  if $\WW$ is a set of convex sets in $\RR^n$,  a shading $Y$ is a subset $Y(W) \subset W$ for each $W \in \WW$.    If $\WW$ is equipped with a shading,  we write $U(\WW) = \bigcup_{W \in \WW} Y(W)$.    If we want to emphasize the shading,  we can write $U(\WW, Y)$.

We can generalize Theorem \ref{thmkak1} to allow shadings,  as long as each $Y(W)$ is large.  If $\WW$ has a shading $Y$,  we define

\begin{equation} \label{deflam}
\lambda(Y) := \min_{W \in \WW} \frac{ |Y(W)| }{|W|}
\end{equation}

\begin{theorem} \label{thmkak} (Kakeya conjecuture, Wang-Zahl, \cite{WZ}) Suppose that $\TT$ is a set of $\delta \times \delta \times 1$-tubes in $B_1 \subset \RR^3$ with a shading $Y$.  If  $ \Delta_{max}(\TT) \lessapprox 1,$ and $\lambda(Y) \gtrapprox 1$,   then 

$$ |U(\TT, Y)| \gtrapprox | \TT| |T|. $$

\end{theorem}

Remark.  This small generalization of Theorem \ref{thmkak1} implies the Hausdorff dimension version of the Kakeya conjecture.   A Kakeya set in $\RR^n$ is a set containing a unit line segment in every direction.   Theorem \ref{thmkak} implies that a Kakeya set in $\RR^3$ has Hausdorff dimension 3.

The proof of Theorem \ref{thmkak} involves a complex induction, and shadings help to make the induction work,  and so it is actually crucial to prove the slightly more general Theorem \ref{thmkak}. 

\subsection{The goal of these notes}

The previous papers of Wang and Zahl, \cite{WZ1} and \cite{WZ2},  prove the sticky case of the Kakeya conjecture.  We will recall the precise statement below.  The paper \cite{WZ} gives a complex inductive argument which reduces the general case of Theorem \ref{thmkak} to the sticky case.  The goal of these notes is to give a detailed outline of this reduction.

We will try to explain and motivate the main ideas.  We will also give a complete outline,  which includes all the cases of the argument.   On the other hand,  we will leave out a lot of details related to pigeonholing, keeping track of shadings,  and keeping track of small parameters.  In Section \ref{secleftout},  we briefly comment on what is left out.  In particular,  we will use $\lessapprox$ and $\ll$ in an imprecise way.   We write $A \lessapprox B$ to mean ``$A$ is not much bigger than $B$'' and $A \ll B$ to mean ``$A$ is much smaller than $B$''.   

We outline here a small variation of the proof in \cite{WZ}.  We have made small adjustments to streamline the proof.   

{\bf Acknowledgements.}
Some of these adjustments were worked out in conversation with Josh and Hong.   It was helpful and fun to talk about the Kakeya proof with Josh and Hong and many others at Oberwolfach in July 2025.  Also,  thanks to Jacob Reznikov for the figures.

\section{Multi-scale structure of $\TT$ and statement of sticky Kakeya}

Our main goal is reduce the general case of Theorem \ref{thmkak} to the sticky case.   So in this section,  we recall the statement of the sticky case.   Before doing that, we set up some language to describe how the tubes of $\TT$ are grouped at different scales.   We will need this language all through the paper.

We will consider sets $\TT$ that are uniform in the following sense.   

\begin{defn} Suppose that $\TT$ is a set of $\delta$-tubes in $B_1 \subset \RR^n$.   
We say that $\TT$ is $\delta^\eps$-uniform if,  for every scale $\rho$ of the form $\delta^{j \epsilon}$,  $j = 1, ..., \epsilon^{-1}$,  there is a set of $\rho$-tubes $\TT_\rho$ so that

\begin{itemize}

\item  $\TT = \bigcup_{T_\rho \in \TT_\rho} \TT[T_\rho]$.

\item The tubes $\TT_\rho$ are essentially distinct.  Therefore each $T \in \TT$ lies in $T_\rho$ for $\sim 1$ choice of $T_\rho \in \TT_\rho$.

\item $| \TT[T_\rho] |$ is roughly constant as $T_\rho$ varies in $\TT_\rho$.

\end{itemize}

\end{defn}

For any $\epsilon > 0$,  it is straightforward to reduce Theorem \ref{thmkak} to the case that $\TT$ is $\delta^\eps$-uniform.

Working with uniform sets is itself an important idea in recent work in the field.
To describe the spacing of a given uniform set of tubes $\TT$,  we can record $| \TT [T_\rho] |$ for each $\rho$ of the form $\rho = \delta^{j \epsilon}$.   This is a lot of information!  In earlier work in geometric measure theory,  people generally described how a set was spaced using just one or two numbers,  such as the Hausdorff dimension of the set.   Recent work in projection theory and geometric measure theory has crucially used the extra information from the complete list of $| \TT [T_\rho]|$.   This philosophy played an important role in the work of Orponen,Shmerkin, and Keleti on the Falconer problem,  and in the solution of the Furstenberg set problem by Orponen, Shmerkin, Ren, and Wang. 

In general,  working with many different scales $\rho$ is crucial to the proof.   The paper \cite{KLT} used two different scales,  $\delta$ and $\delta^{1/2}$.   The proof of Kakeya and the earlier proof of Furstenberg involve considering a huge number of scales $\rho$.   Of course in principle considering more scales gives us more information to work with.   However,  coordinating all that information is a real challenge.    A major part of the proof of Furstenberg and the proof of Kakeya is understanding how to coordinate all these different scales.   While working with multiple scales is a classical topic in harmonic analysis,  these proofs involve novel and subtle ways of doing so. 

The sticky case refers to a uniform set where all scales have similar spacing.   Similar special cases appear in many problems in geometric measure theory.  In other areas,  this special case is often called the (almost) Ahlfors-David regular case.  Here is the definition.

\begin{defn}
We say that $\TT$ is $\delta^\eps$-sticky if $\TT$ is $\delta^\eps$-uniform and for every scale $\rho$ of the form $\delta^{j \epsilon}$, 

$$\delta^\eps \left( \frac{\rho}{\delta} \right)^{2} \le   | \TT[T_\rho] | \le \delta^{-\eps} \left( \frac{\rho}{\delta} \right)^{2}. $$

\end{defn}

The sticky case first appeared in work of Wolff (unpublished) and Katz-Laba-Tao \cite{KLT}.   In the early 2000s,  Katz and Tao developed an approach to the sticky case.   The sticky case of Kakeya was proven  by Wang and Zahl in \cite{WZ1} and \cite{WZ2}, following the general approach of Katz-Tao and dealing with significant technical issues.  They proved the following theorem.

\begin{theorem} (Sticky Kakeya,  \cite{WZ2} building on \cite{WZ1}) \label{thmstickykak}
For every $\alpha  > 0$,  there is $\exsticky(\alpha)
> 0$ so that the following holds.   Suppose $\TT$ is a set of $\delta$ tubes in $B_1 \subset \RR^3$ so that

\begin{itemize}

\item  $\TT$ is $\delta^{\exsticky}$-sticky

\item $\Delta_{max}(\TT) \lessapprox \delta^{-\exsticky}$

\item $\TT$ has a shading with $\lambda(\TT) \ge \delta^{\exsticky}$.   

\end{itemize}

Then 

$$ |U(\TT)| \gtrapprox |\TT|^{-\alpha} |\TT| |T|. $$

\end{theorem}

The sticky case is the special case where the spacing is similar at all scales.   While it sounds like an important case,  it is also a very special case.   It was not clear for a long time whether the sticky case helps significantly in understanding the general case.    But it turns out that it does.

The sticky case plays a crucial role in the proof of the Kakeya set conjecture.   Earlier,  the sticky case of the Furstenberg set conjecture (proven by Orponen and Shmerkin) played a key role in the full proof of the Furstenberg set conjecture (by Ren and Wang). 


Reducing to the sticky case played a key role in two major results in the field (Kakeya in $\RR^3$ and Furstenberg set conjecture in $\RR^2$).   It looks like it will be a crucial approach in harmonic analysis / geometric measure theory going forward.   Comparing the reduction to the sticky case in these two major problems, there are some related high level ideas but there are also significant differences.  
There is a lot to explore about how reducing to the sticky case might work in other problems in the area.

\section{First glimmer of reducing to the sticky case} \label{secglimmer}

If a set of tubes has $|U(\TT)| \ll | \TT| |T|$, then a typical point in the union must lie in many tubes.  It can be helpful to consider the typical multiplicity.  
In general, we define the multiplicity of a set of convex sets $\WW$ with a shading $Y$ as

\begin{equation} \label{defmu} \mu(\WW) := \frac{ \sum_{W \in \WW} |Y(W)| }{| U(\WW, Y)|}. 
\end{equation}

Information about $\mu(\WW)$ is equivalent to information about $|U(\WW)|$, but it can be more intuitive to think about one than the other.
We can rephrase Theorem \ref{thmkak} in terms of multiplicity as follows.  If $\TT$ is a set of $\delta$-tubes in $\RR^3$ with $\Delta_{max}(\TT) \lessapprox 1$ and $\lambda(Y) \gtrapprox 1$,  then $\mu(\TT) \lessapprox 1$.

The proof is based on a bootstrapping argument where we start with some weaker Kakeya bounds and use them to prove stronger Kakeya bounds.  We set this up as follows.

\begin{defn} We write $K(\beta)$ for the following statement.   If $\TT$ is a set of $\delta$-tubes in $B_1 \subset \RR^3$ so that

\begin{itemize}

\item  $ \Delta_{max}(\TT) \lessapprox 1, $

\item $Y$ is a shading of $\TT$ with $|Y(T)| \gtrapprox |T|$.   

\end{itemize}

then $ \mu( \TT, Y) \lessapprox | \TT |^\beta. $

\end{defn}

The statement $K(1)$ is trivial,  and Theorem \ref{thmkak} is equivalent to knowing $K(\beta)$ for every $\beta > 0$.   We will prove Theorem \ref{thmkak1} by a bootstrapping argument,  where we show that for every $\beta > 0$,  $K(\beta)$ implies $K(\beta - \nu)$ for some small $\nu > 0$.   

In this section,  we give a first glimmer of how such an argument may look. 

We will assume $K(\beta)$.   We suppose $\Delta_{max}(\TT) \lessapprox 1$,  and we have to show that  $\mu(\TT) \ll | \TT|^{\beta}$,  meaning that $\mu(\TT)$ is much smaller than $|\TT|^{\beta}$.   In this section,  we focus on the case that $| \TT | \sim \delta^{-2}$,  and so our goal is to prove that 

$$\mu(\TT) \ll | \TT |^\beta = \delta^{- 2 \beta}. $$

We can also write our goal in terms of $|U(\TT)|$.   When $| \TT |\sim \delta^{-2}$,  $K(\beta)$ implies $|U(\TT)| \gtrapprox \delta^{2 \beta}$,  and so our goal is to prove

$$ |U(\TT)| \gg \delta^{2 \beta}. $$

If $\TT$ is sticky,  then our goal follows by sticky Kakeya.   So we can assume that $\TT$ is not sticky.  We have to take advantage of the fact that $\TT$ is not sticky in order to get a small improvement on the bound coming directly from $K(\beta)$.  


A crucial part of the argument is to study how a set of $\delta \times \delta \times 1$ tubes $\TT$ interacts with a small ball $B = B(x,r) \subset B(0,1)$.   Figure \ref{fig:tube_int_ball_2} shows how tubes of $\TT$ may intersect a small ball $B$:

   \begin{figure}
      \begin{center}
        \tikzset{every picture/.style={line width=0.75pt}} 

\begin{tikzpicture}[x=0.75pt,y=0.75pt,yscale=-1,xscale=1]

\draw  [color={rgb, 255:red, 74; green, 144; blue, 226 }  ,draw opacity=0.48 ] (272.96,42.94) -- (149.17,235.18) -- (144.96,232.47) -- (268.75,40.24) -- cycle ;
\draw  [color={rgb, 255:red, 74; green, 144; blue, 226 }  ,draw opacity=0.48 ] (264.17,35.13) -- (163.44,240.39) -- (158.95,238.18) -- (259.68,32.92) -- cycle ;
\draw  [color={rgb, 255:red, 74; green, 144; blue, 226 }  ,draw opacity=0.48 ] (293.43,217.98) -- (152.57,37.88) -- (156.51,34.8) -- (297.37,214.9) -- cycle ;
\draw  [color={rgb, 255:red, 74; green, 144; blue, 226 }  ,draw opacity=0.48 ] (303.7,212.25) -- (142.96,49.65) -- (146.51,46.13) -- (307.25,208.74) -- cycle ;
\draw  [color={rgb, 255:red, 74; green, 144; blue, 226 }  ,draw opacity=0.48 ] (231.25,240.32) -- (211.01,12.57) -- (215.99,12.13) -- (236.23,239.87) -- cycle ;
\draw  [color={rgb, 255:red, 74; green, 144; blue, 226 }  ,draw opacity=0.48 ] (242.99,241.09) -- (196.55,17.22) -- (201.44,16.2) -- (247.89,240.08) -- cycle ;
\draw  [color={rgb, 255:red, 74; green, 144; blue, 226 }  ,draw opacity=0.48 ] (128.16,156) -- (352.35,111.08) -- (353.33,115.98) -- (129.15,160.9) -- cycle ;
\draw  [color={rgb, 255:red, 74; green, 144; blue, 226 }  ,draw opacity=0.48 ] (128.67,167.75) -- (346.16,97.21) -- (347.7,101.96) -- (130.21,172.5) -- cycle ;
\draw   (184,138.17) .. controls (184,113.22) and (204.22,93) .. (229.17,93) .. controls (254.11,93) and (274.33,113.22) .. (274.33,138.17) .. controls (274.33,163.11) and (254.11,183.33) .. (229.17,183.33) .. controls (204.22,183.33) and (184,163.11) .. (184,138.17) -- cycle ;
\draw  [color={rgb, 255:red, 220; green, 50; blue, 71 }  ,draw opacity=1 ][fill={rgb, 255:red, 220; green, 50; blue, 71 }  ,fill opacity=0.41 ] (194.76,159.24) -- (231.82,94.56) -- (235.29,96.55) -- (198.23,161.22) -- cycle ;
\draw  [color={rgb, 255:red, 220; green, 50; blue, 71 }  ,draw opacity=1 ][fill={rgb, 255:red, 220; green, 50; blue, 71 }  ,fill opacity=0.41 ] (205.65,101.71) -- (262.81,167.46) -- (259.79,170.08) -- (202.63,104.33) -- cycle ;
\draw  [color={rgb, 255:red, 220; green, 50; blue, 71 }  ,draw opacity=1 ][fill={rgb, 255:red, 220; green, 50; blue, 71 }  ,fill opacity=0.41 ] (220.84,95) -- (233.04,181.26) -- (229.08,181.82) -- (216.88,95.56) -- cycle ;
\draw  [color={rgb, 255:red, 220; green, 50; blue, 71 }  ,draw opacity=1 ][fill={rgb, 255:red, 220; green, 50; blue, 71 }  ,fill opacity=0.41 ] (271.48,129.82) -- (187.06,151.34) -- (186.07,147.47) -- (270.49,125.95) -- cycle ;

\end{tikzpicture}
      \end{center}
      \caption{Tubes of $\TT$ intersecting a small ball $B$}\label{fig:tube_int_ball_2}
    \end{figure}
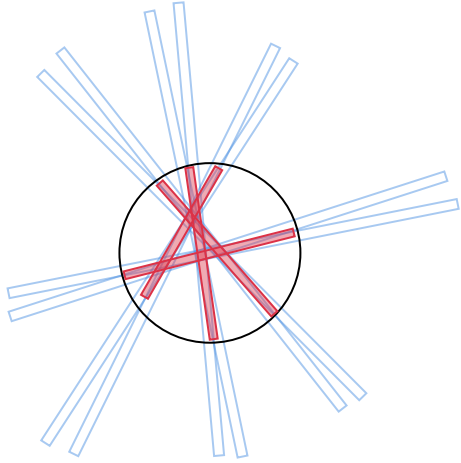

The tubes of $\TT$ are the long blue tubes.   Each long blue tube $T$ intersects $B$ in a shorter red tube $T_B$.   We let $\TT_B$ denote this set of shorter red tubes.    A typical short tube $T_B \in \TT_B$ may be contained in many different tubes $T \in \TT$.   

We define

$$ \TT_{T_B} := \{ T \in \TT: T_B \subset 2T \}. $$

Since $B$ has radius $r$,  the angle between two tubes in $\TT_{T_B}$ is $\lesssim \delta/r$.   We set $\rho = \delta/r$.   So we see that all the tubes of $\TT_{T_B}$ are contained in a single $\rho$-tube $T_\rho \in \TT_\rho$.   Therefore,  we see that 

$$ | \TT_{T_B} | \lessapprox \mu(\TT[T_\rho]).  $$

Now we can bound $\mu(\TT)$ by

\begin{equation} \label{multprodwarmup}
 \mu(\TT) \lessapprox \mu(\TT[T_\rho]) \mu(\TT_B). 
 \end{equation}

Since $\TT$ is not sticky,  we can assume that there is a scale $\rho$ where $| \TT(T_\rho)| \ll (\delta/\rho)^{-2} $.   We choose this value for $\rho$ and examine the factors in (\ref{multprodwarmup}).   Since $\TT[T_\rho] \subset \TT$,  we have $\Delta_{max}(\TT[T_\rho]) \lessapprox 1$,  and so $K(\beta)$ gives us the bound

\begin{equation} \label{mult1}
 \mu(\TT[T_\rho]) \lessapprox | \TT[T_\rho] |^\beta \ll (\delta/\rho)^{- 2 \beta}. 
 \end{equation}

Next we have to bound $\mu(\TT_B)$.   Here it is not clear whether $\Delta_{max}(\TT_B) \lessapprox 1$,  and so it is not clear whether we can use $K(\beta)$ directly.   Let us consider the special case when 

\begin{equation} \label{KTcase}
\Delta_{max}(\TT_B) \lessapprox 1.
\end{equation}

In this special case,  we would have $\mu(\TT_B) \lessapprox | \TT_B |^\beta$.   We would also have $|\TT_B| \lessapprox \frac{|B|}{|T_B|} \sim \rho^{-2}$.   All together this would give the bound

\begin{equation} \label{mult2}
 \mu(\TT_B) \lessapprox |\TT_B|^\beta \lessapprox \rho^{-2 \beta}
\end{equation}

Combining (\ref{mult1}) and (\ref{mult2}) would give

$$ \mu(\TT) \lessapprox \mu(\TT[T_\rho]) \mu(\TT_B) \ll (\delta/\rho)^{-2 \beta} \rho^{-2 \beta} = \delta^{-2 \beta}. $$

This proves our desired bound in the special case when $\Delta_{max}(\TT_B) \lessapprox 1$. 

This sounds nice,  but it was just wishful thinking to suppose that $\Delta_{max}(\TT_B) \lessapprox 1$.   We don't know anything about the set of tubes $\TT_B$.   As far as we know,  $\TT_B$ could even contain all possible $\delta \times \delta \times r$ tubes in $B$.   

Let us next consider the special case when $\TT_B$ consists of all tubes in $B$ -- or more precisely,  when $\TT_B$ is a maximal set of essentially distinct $\delta \times \delta \times r$  tubes in $B$.   This case is also good for us for a different reason.   In this case,  we have $|U(\TT_B)| \sim |B|$.   Since $B$ represented a typical $r$-ball in the $r$-neighborhood of $U(\TT)$,  we would then have

$$ |U(\TT) | \gtrapprox |U(\TT_r)|. $$

The tubes $T_r$ are thicker than the original tubes $T$,  and intuitively, this should help us prove a good lower bound for $U(\TT_r)$.   Our goal inequality,  $\mu(\TT) \ll \delta^{2 \beta}$,  is equivalent to the bound

\begin{equation} \label{volgoal}
|U(\TT)| \gg \delta^{2 \beta}
\end{equation}

We cannot immediately apply $K(\beta)$ to $U(\TT_r)$.   The tubes $\TT_r$ don't necessarily obey $\Delta_{max}(\TT_r) \lessapprox 1$.    The inequality $\Delta_{max}(\TT) \lessapprox 1$ tells us that $| \TT [T_r] | \lessapprox (r/\delta)^2$.   Since $| \TT | \sim \delta^{-2}$,  we get $| \TT_r | \gtrapprox r^{-2}$.    If $| \TT_r  | \gg r^{-2}$,  then $\Delta_{max}(\TT_r) \ge \Delta(\TT_r, B_1) \gg 1$.   To get around this,  define $\tilde \TT_r \subset \TT_r$ to be a random subset of $\sim r^{-2}$ tubes.   We see right away that $\Delta(\tilde \TT_r, B_1) \lesssim 1$.   It is not hard to check that for any convex set $K \subset B_1$,  $\Delta(\tilde \TT_r, K) \lessapprox 1$,  and so $\Delta_{max}(\tilde \TT_r) \lessapprox 1$.    The details of this argument appear in Section \ref{subsecdenssmallballs}.

Now $K(\beta)$ tells us that $|U(\tilde \TT_r)| \gtrapprox r^{2 \beta}$.  Therefore,  we have

$$ |U(\TT)| \gtrapprox |U(\TT_r)| \ge |U(\tilde \TT_r)| \gtrapprox r^{2 \beta} \gg \delta^{2 \beta}. $$

This gives us our desired bound in the special case when $\TT_B$ is full -- when $\TT_B$ contains all the tubes in $B$.  

So far we have gotten our desired goal in two rather opposite special cases: when $\Delta_{max}(\TT_B) \lessapprox 1$,  and when $\TT_B$ is full.   So we can ask whether this inductive argument can be extended to cover all possible cases.

The proof of Wang and Zahl does indeed extend this argument to cover all possible cases,  although it will require significant additional ideas.

One important case we still have to address is when $\TT_B$ is concentrated in a planar slab.  In this case,  $\Delta_{max}(\TT_B)$ is large and yet $|U(\TT_B)| \ll |B|$.   More generally,  if we want to extend this argument to cover ``all cases'',  we need a good way to organize all the cases.

In the next section,  we discuss  how to organize an arbitrary set of tubes $\TT$,  or more generally an arbitrary set of convex sets.

\section{Organizing convex sets}

Recall from the last section that if $\Delta_{max}(\TT_B) \lessapprox 1$,  then we get our desired goal.   So we are concerned with the case when $\Delta_{max}(\TT_B) \gg 1$.   It turns out to be useful to consider the convex sets $W$ so that $\Delta(\TT_B,  W)$ is near maximal.   It is helpful to organize $\TT_B$ using these convex sets. 

This organizational strategy applies not just to tubes in $\RR^3$ but to a set of convex sets in $\RR^n$.   We describe it in the following lemma.  

\begin{lemma} \label{lemmafactmax} (Maximal density factoring lemma) Suppose that $\VV$ is a finite set of convex sets in $\RR^n$.    Then there is a subset $\VV' \subset \VV$ with $|\VV'| \gtrapprox |\VV|$ and a set of convex sets $\WW$ so that

\begin{itemize}

\item For each $W \in \WW$,  there is a set $\VV_{uni}[W] \subset \VV[W]$ so that

$$ \VV' = \bigsqcup_{W \in \WW} \VV_{uni}[W]. $$

\item For each $W \in \WW$,  $\Delta(\VV_{uni}[W], W) \approx \Delta_{max}(\VV')$.

\item For each $W \in \WW$,   $| \VV_{uni}[W] | \approx | \VV'[W] |$.

\item For each $W \in \WW$,  $C_F( \VV_{uni}[W],  W) \le 1$.  

\item $\Delta_{max}(\WW) \lessapprox 1$.

\end{itemize}

\end{lemma}

\begin{proof} We choose sets $W_j$ one at a time by a greedy algorithm.

Choose $W_0$ to maximize $\Delta(\VV, W_0)$.   Then set $\VV_1 = \VV \setminus \VV[W_0]$.    Now choose $W_1$ to maximize $\Delta(\VV_1, W_1)$.   Continue in this way until $\VV_s$ is empty.   Then stop.   

For each $j$,  we set $\VV_{uni}[W_j] = \VV_{j}[W_j]$.  We note that $\VV_{uni}[W_j]$ are disjoint.

Next we pigeonhole $\Delta(\VV_{j},  W_j)$.   We set $J(\lambda) = \{ j: \Delta(\VV_{j}, W_j) \sim \lambda \}$.   We set $\VV_\lambda = \cup_{j \in J(\lambda)} \VV_{uni}[W_j]$.   We note that $\VV = \sqcup_{\lambda} \VV_\lambda$.   We choose $\lambda$ so that $|\VV_\lambda| \approx |\VV|$ and we set $\VV' = \VV_{\lambda}$.   

Now we set $\WW = \{ W_j \}_{j \in J(\lambda)}$,  so that $\VV' = \VV_\lambda = \sqcup_{W \in \WW} \VV_{uni} [W]$.  This gives the first bullet point.

Note that $\Delta(\VV_{j},  W_j) = \Delta_{max}(\VV_{j})$ is non-increasing.  Set $j_1$ to be the first number in $J(\lambda)$.    Since $\VV' \subset \VV_{j_1}$ and $\VV_{uni}[W_j] \subset \VV'$, we have

$$ \Delta_{max}(\VV') \le \Delta_{max}(\VV_{j_1}) = \Delta(\VV_{uni}[W_{j_1}], W_{j_1}) \le \Delta_{max}(\VV'). $$

Therefore,  all the inequalities above are equalities.  

Now for each $j \in J(\lambda)$ we have

\begin{equation} \label{factmaxappeq}
\Delta_{max}(\VV_{j_1})\ge \Delta(\VV_{j}, W_j) = \Delta(\VV_{uni}[W_j], W_j) = \Delta_{max}(\VV_{j}) \sim \Delta_{max}(\VV_{j_1}). 
\end{equation}

\noindent Therefore,  all the inequalities in the above are approximate equalities,  which shows that 

$$\Delta(\VV_{uni}[W_j], W_j) \sim \Delta_{max}(\VV_{j_1}) = \Delta_{max}(\VV'). $$

\noindent This gives the second bullet.   These approximate equalities also show that $\Delta(\VV'[W_j], W_j) \le \Delta_{max}(\VV') \sim \Delta(\VV_{uni}[W_j],  W_j)$,  which implies that $|\VV_{uni}[W_j]| \approx |\VV'[W_j]|$.   This gives the third bullet.

The fourth bullet point follows because we chose each $W_j$ to maximize $\Delta(\VV_{j}, W_j)$.   For each $W_j$,  we have

$$ \Delta(\VV_{uni}[W_j], K) \le \Delta(\VV_{j}, K) \le \Delta(\VV_{j}, W_j) = \Delta(\VV_{uni}[W_j], W_j). $$

The intuition behind the fifth bullet is that if $\Delta(\WW, U) \gg 1$,  then $\Delta(\VV', U) \gg \Delta(\VV', W) \approx \Delta_{max}(\VV')$.    First note that for each $W \in \WW$,

$$ \sum_{V \in \VV_{uni}[W]} |V| = \Delta(\VV_{uni}[W], W) |W| \gtrapprox \Delta_{max}(\VV') |W|. $$

Therefore we have

$$ \sum_{V \in \VV'[U]} |V| \ge \sum_{W \in \WW[U]} \sum_{V \in \VV_{uni}[W]} |V| \gtrapprox \Delta_{max}(\VV') \sum_{W \in \WW[U]} |W| = \Delta_{max}(\VV') \Delta(\WW, U) |U|. $$

Dividing both sides by $|U|$,  we get

$$ \Delta(\VV', U) \gtrapprox \Delta_{max}(\VV') \Delta(\WW, U). $$

And so $\Delta(\WW, U) \lessapprox \frac{ \Delta(\VV', U)}{\Delta_{max}(\VV')} \le 1.$

\end{proof}

Remark.  The distinction between $\VV$ and $\VV'$ is a minor technical issue.   And the distinction between $\VV_{uni}[W]$ and $\VV'[W]$ and $\VV[W]$ is a minor technical issue.  In this outline,  we will ignore these minor technical issues and pretend that $\VV = \VV'$ and that $\VV_{uni}[W] = \VV[W]$.  

\vskip10pt

To get some intuition for this lemma,  let us mention a couple of examples.   

Suppose that $\TT$ is a set of disjoint tubes.   Then $\Delta_{max}(\TT) = 1$.   In this case,  the set $\WW$ given by the lemma is $\WW = \TT$.

Suppose that $\TT$ consists of all the $\delta$ tubes in $B_1 \subset \RR^3$.   (Strictly speaking,  we mean that $\TT$ is a maximal set of essentially disjoint $\delta$-tubes in $B_1 \subset \RR^3$.)  Then we have $|\TT| \sim \delta^{-4}$,  and so $\Delta(\TT,  B_1) \sim \delta^{-2}$.   It is not hard to check that $\Delta(\TT,  W)$ is essentially maximized by taking $W = B_1$.   Therefore,  in this case,  the set $\WW$ given by the lemma is $\WW  = \{ B_1 \}$.

We call the set $\WW$ coming from Lemma \ref{lemmafactmax} the maximal density factoring of $\VV$.   And we call Lemma \ref{lemmafactmax} the maximal density factoring lemma.  

Remark.  Lemma \ref{lemmafactmax} is stated for convex sets,  but convexity did not play any role in the proof.   We can replace convex sets by other classes of sets.   In general,  we can let $\XX$ denote any class of measurable sets in a measure space $\Omega$.   If $\VV \subset \XX$ and $X \in \XX$,  then we define

$$ \Delta(\VV,  X) = \frac{ \sum_{V \in \VV[X]} |V| }{|X|}, $$

$$ \Delta_{max,  \XX}(V) = \sup_{X \in \XX} \Delta(\VV,  X). $$

\noindent Then we can make a version of Lemma \ref{lemmafactmax} for sets from $\XX$ instead of convex sets.   In the statement,  we replace the convexity condition by the condition $\VV \subset \XX$ and $\WW \subset \XX$,  and we replace $\Delta_{max}$ by $\Delta_{max, \XX}$.    For instance,  we can replace general convex sets by slabs with dimensions $s \times 1 \times 1$.   Or we can replace convex sets by semi-algebraic sets of bounded complexity.  

We use the word factoring because these lemmas help us to break the problem of understanding $\mu(\VV)$ into two smaller problems: understanding $\mu(\VV[W])$ and understanding $\mu(\WW)$.   By pigeonholing,  we can reduce to the case that $\mu(\VV[W])$ is essentially constant.   Then it looks intuitive that $\mu(\VV) \lessapprox \mu(\VV[W]) \mu(\WW)$.  Because $\mu(\VV)$ represents the average multiplicity of $\VV$ and not the maximal multiplicity of $\VV$,  we have to be a little bit careful about this inequality.   The inequality does not hold in complete generality,  but it does hold in all the situations we need it.   We discuss this briefly in Section \ref{secleftout}.


Let us now discuss what we gain from the maximal density factoring lemma.   We began with an arbitrary set of convex sets $\VV$.   The set $\VV$ does not obey any ``non-clustering condition'' such as $\Delta_{max}(\VV) \lessapprox 1$.   Using the factoring lemma,  we can often reduce the problem of understanding $\VV$ to the problem of understanding $\VV[W]$ and the problem of understanding $\WW$.   Each of these sets does obey a useful ``non-clustering condition''.    Let us separate out these two conditions and give them names.  

\begin{defn} A set of convex sets $\VV$ is called $C$-convex-Katz-Tao if $\Delta_{max}(\VV) \le C$.  We say $\VV$ is convex Katz-Tao if $\Delta_{max}(\VV) \lessapprox 1$.
\end{defn}

Notice that if $\WW$ is a maximal density factoring of $\VV$,  then $\WW$ is Katz-Tao.

If $\VV$ is a set of convex sets in $K$, we define

\begin{equation} \label{defcf} C_F(\VV, K) = \frac{ \max_{K' \subset K} \Delta(\VV, K') }{\Delta(\VV,K)}. \end{equation}

\begin{defn} If $\VV$ is a set of convex sets all contained in a convex set $K$, we say that $\VV$ is $C$-convex-Frostman in $K$ if $C_F(\VV, K) \le  C$.   We say $\VV$ is convex Frostman in $K$ if  $C_F(\VV, K) \lessapprox 1$.\end{defn}

Notice that if $\WW$ is a maximal density factoring of $\VV$,  then $\VV_{uni}[W]$ and $\VV[W]$ are Frostman in $W$,  while $\WW$ is Katz-Tao.

Using the maximal density factoring,  we can often reduce the problem of understanding an arbitrary collection of convex sets to the problem of understanding the Katz-Tao case and the problem of understanding the Frostman case.  

We have already been discussing the Katz-Tao case,  and the Frostman case is something new.   We take a little time here to digest the Frostman case.
Suppose for example that $\TT$ is Frostman in $B_1$.   This means that $U(\TT) \subset B_1$ and that $\Delta(\TT,  B_1) \gtrapprox \Delta(\TT, K)$ for any convex $K \subset B_1$.   It may well happen that $\Delta(\TT, B_1) \gg 1$.   For example,  if $\TT$ is the set of all $\delta$-tubes in $B_1$,  then $\TT$ is Frostman in $B_1$.   It may also happen that $\Delta(\TT, K) \gg 1$ for some convex set $K \subset B_1$.   But if $\TT$ is Frostman in $B_1$,  then $\Delta(\TT, K) \lessapprox \Delta(\TT, B_1)$,  and so the set $K$ cannot contain `more than its fair share of tubes'.   Another good example of a Frostman set to keep in mind is a random set of $N$ $\delta$-tubes $T$ in $B_1$ where $N$ is chosen so that $N |T| \gg |B_1|$.   

\section{Main bootstrapping lemmas}

Notice that the Kakeya bound $K(\beta)$ is a bound for the Katz-Tao case.   From now on,  we will call it $K_{KT}(\beta)$.   We recall it here, and then we introduce an analogous bound for the Frostman case.   Then we can state the two main bootstrapping lemmas in the proof of Theorem \ref{thmkak}.


\begin{defn} We write $K_{KT}(\beta)$ for the following statement.   If $\TT$ is a set of $\delta$-tubes in $B_1 \subset \RR^3$ so that

\begin{itemize}

\item  $ \Delta_{max}(\TT) \lessapprox 1, $

\item $Y$ is a shading of $\TT$ with $\lambda(Y) \gtrapprox 1$.   

\end{itemize}

then $ \mu( \TT, Y) \lessapprox | \TT |^\beta. $

\end{defn}

Note that $K_{KT}(1)$ is trivial.   It is not hard to prove $K_{KT}(\beta)$ for some $\beta$ slightly less than 1.   Our goal is to prove $K_{KT}(0)$.  

One important case is when $| \TT | \sim \delta^{-2}$.   In this case,  $K_{KT}(\beta)$ gives $\mu(\TT) \lessapprox \delta^{-2 \beta}$,  which is equivalent to $|U(\TT)| \ge \delta^{2 \beta}$.  

\begin{defn} We write $K_{F}(\beta)$ for the following statement.   If $\TT$ is a set of $\delta$-tubes in $B_1 \subset \RR^3$,  and

\begin{itemize}

\item  $C_F(\TT, B_1) \lessapprox 1$. 

\item $Y$ is a shading of $\TT$ with $\lambda(Y) \gtrapprox 1$.,  

\end{itemize}

then $ \mu( \TT) \lessapprox ( \delta^{-2})^\beta \left( \delta^2 | \TT | \right)^{1 - \beta}. $

Or equivalently,

$$ |U(\TT)| \gtrapprox \delta^{2 \beta} \left( \delta^2 | \TT| \right)^{\beta}.$$

\end{defn}

Let us parse the algebra in this bound.   Since $C_F(\TT, B_1) \lessapprox 1$,  we have $| \TT | \gtrapprox \delta^{-2}$.    One important case is when $| \TT | \sim \delta^{-2}$.   In this case,  the bound $C_F(\TT, B_1) \lessapprox 1$ is equivalent to the bound $\Delta_{max}(\TT) \lessapprox 1$.   So $\TT$ is both Frostman and Katz-Tao.   In this case,   $K_{KT}(\beta)$ gives $|U(\TT)| \gtrapprox \delta^{2 \beta}$,  and $K_F(\beta)$ gives the same bound $|U(\TT)| \gtrapprox \delta^{2 \beta}$.    When $| \TT | \gg \delta^{-2}$,  then $K_F(\beta)$ gives a stronger bound,  because the last factor $ \left( \delta^2 | \TT| \right)^{\beta}$ is $\gg 1$.   The exact exponent $\beta$ in this final factor is not crucial in our analysis,  and the argument would work equally well with any positive exponent.   But the exponent $\beta$ is natural here -- for example if $\TT$ consists of all $\delta$-tubes in $B_1$ then $|\TT| \sim \delta^{-4}$ and $K_F(\beta)$ gives the sharp upper bound $\mu(\TT) \lessapprox \delta^{-2}$.  

Comparing $K_{KT}(\beta)$ and $K_F(\beta)$,  neither immediately implies the other.   When $| \TT | \sim \delta^{-2}$ and $\Delta_{max}(\TT) \lessapprox 1$,  $K_{KT}(\beta)$ and $K_F(\beta)$ give the same bound.   The statement $K_F(\beta)$ gives extra information when $\TT$ is Frostman and $| \TT | \gg \delta^{-2}$.   And the statement $K_{KT}(\beta)$ gives extra information when $\TT$ is Katz-Tao and $| \TT | \ll \delta^{-2}$.  

The two main lemmas in the proof are

\begin{mlem} \label{lemmain1} $K_{KT}(\beta)$ implies $K_F(\beta)$.  
\end{mlem}

\begin{mlem} \label{lemmain2}  If $K_{KT}(\beta)$ and $K_{F}(\beta)$ hold,  then $K_{KT}(\beta - \nu)$ holds for $\nu = \nu(\beta) > 0$.   
\end{mlem}

Since the bound $K_{KT}(1)$ is trivial,  the two main lemmas imply that $K_{KT}(\beta)$ and $K_F(\beta)$ hold for all $\beta > 0$,  which implies the Kakeya conjecture,  Theorem \ref{thmkak}.

The proof of each main lemma involves a novel multiscale argument and uses the sticky Kakeya theorem,  Theorem \ref{thmstickykak}.

We note that $K_{KT}(\beta)$  implies some bounds for sets of tubes with any value of $\Delta_{max}(\TT)$.  

\begin{lemma} \label{genKKT} Suppose that $K_{KT}(\beta)$ holds.   Suppose that $\TT$ is a set of $\delta$-tubes in $B_1$ with a shading with $\lambda(\TT) \gtrapprox 1$.   Then

$$ \mu( \TT) \lessapprox \Delta_{max}(\TT)^{1 - \beta}  | \TT |^\beta.$$

\end{lemma}

\begin{proof} Choose a random subset $\TT' \subset \TT$ with cardinality $\frac{1}{\Delta_{max}(\TT)} | \TT|$.   It is straightforward to check that $\Delta_{max}(\TT') \lessapprox 1$,  and so $\mu(\TT') \lessapprox |\TT'|^\beta$.  But then

$$ \mu(\TT) \lessapprox \Delta_{max}(\TT) \mu(\TT') \lessapprox \Delta_{max}(\TT) \left( \frac{| \TT|}{\Delta_{max}(\TT)} \right)^\beta. $$

\end{proof}

\section{Slabs and planks} \label{secslabplank}

In the course of our argument,  we will study not just incidences of tubes but also incidences of slabs and of planks.  Incidences of slabs are well understood by simple $L^2$ methods.   Combining our knowledge of slabs with what we know about tubes gives incidence bounds for planks.   A general convex set in $ \RR^3$ is approximately a rectangular plank of dimensions $a \times b \times c$,  and so we get incidence bounds for arbitrary convex sets in $\RR^3$.  We will see that incidence bounds for arbitrary convex sets in $\RR^3$ can be systematically reduced to incidence bounds for tubes.

\subsection{Slabs}

A standard $L^2$ method gives strong bounds for the incidences of slabs.   A key character is the notion of typical intersection angle.  Suppose that $\SSS$ is a set of $\delta \times 1 \times 1$ slabs with shading.  Define the set of incidence triples

$$Tri(\SSS) :=  \{ (x, S_1, S_2) : x \in Y(S_1) \cap Y(S_2) \}. $$

There is a natural measure on $Tri(\SSS)$ coming from the volume measure for $x$ and the counting measure for $S_1, S_2$.  We can define the angle between two slabs $S_1, S_2$, which is well defined up to additive error $\delta$.  Then we define

$$Tri_\theta(\SSS) :=  \{ (x, S_1, S_2) \in Tri(\SSS): \textrm{angle}(S_1, S_2) \sim \theta \}. $$

We say that $\theta$  is a typical angle of intersection for $\SSS$ if $|Tri_\theta(\SSS)| \approx |Tri(\SSS)|$.   If $\mu(\SSS) \ge 2$, then $\SSS$ always has a typical intersection angle $\theta \in [\delta, 1]$.  

\begin{lemma} \label{slab3} Suppose $\SSS$ is a set of $\delta \times 1 \times 1$ slabs contained in $\tilde S$,  a $\theta \times 1 \times 1$ slab,  with $\delta \le \theta \le 1$.   Suppose the slabs of $\SSS$ are shaded with $\lambda(\SSS) \gtrapprox 1$.   If $\theta$ is a typical intersection angle for $\SSS$, then  $|U(\SSS)| \approx |\tilde S|$.
\end{lemma}

\begin{proof} We let $f = \sum_{S \in \SSS} 1_{Y(S)}$.   Note that 

$$\int f^2 = |Tri(\SSS)| = \sum_{S_1, S_2 \in \SSS} |Y(S_1) \cap Y(S_2)|. $$

By hypothesis, we have

$$ \int f^2 =  | Tri(\SSS) | \approx \sum_{S_1, S_2 \in \SSS, \textrm{angle}(S_1, S_2) \sim \theta } |Y(S_1)\cap Y(S_2)| \le |\SSS|^2 \frac{\delta^2}{\theta}. $$

On the other hand, we know that $\int f \approx |\SSS| |S| = |\SSS| \delta$.  By Cauchy-Schwarz,  since $f$ is supported in $U(\SSS)$,  

$$ \int f^2 \ge \frac{ ( \int f )^2}{|U(\SSS)|} \approx \frac{ |\SSS|^2 \delta^2 }{ |U(\SSS)|}. $$

Comparing,  we get 

$$ |U(\SSS)| \gtrapprox \theta = |\tilde S|. $$

\end{proof}

This $L^2$ bound leads to sharp estimates for the analogue of Theorem \ref{thmkak} where tubes are replaced by slabs.  We state this estimate for context, although we will not use it directly below.

\begin{lemma} \label{slab1} If $\SSS$ is a set of $\delta \times 1 \times 1$ slabs in $B_1$,  and  $\Delta_{max}(\SSS) \lessapprox 1$, then  $ \mu( S) \lessapprox 1. $

\end{lemma}





\subsection{Planks}

Next we suppose that $\PP$ is a set of $a \times b \times 1$ planks in $B_1$.   We will study $\mu(\PP)$ by relating it to incidence problems involving slabs and tubes.   In a sense,  an incidence problem for planks can be broken down into an incidence problem for slabs and an incidence problem for tubes.   The incidence problem for slabs is solved completely using Lemma \ref{slab3}.   Assuming $K_{KT}(\beta)$ or $K_F(\beta)$ we get information about the incidence problem for tubes.   In this way,  we are able to prove estimates for $\mu(\PP)$ based on various geometric assumptions about $\PP$.



Suppose that $\PP$ is a set of $a \times b \times 1$ planks in $B_1$, equipped with a shading, and with $\mu(\PP) \gg 1$.  We will define the typical angle of intersection of planks in $\PP$, analogous to the definition above for slabs.  Recall that a plank $P$ has dimensions $a \times b \times 1$ with $a \le b \le 1$.  We let $TP$ be the 2-plane spanned by the two long sides of $P$.  We define the angle between two planks $P_1, P_2$ to be the angle between $TP_1, TP_2$.  It is well defined up to an additive error of $a/b$.  We consider as above the incidence triples

$$ Tri(\PP) := \{ (x, P_1, P_2): x \in Y(P_1) \cap Y(P_2) \}. $$
\noindent As above, $Tri(\PP)$ has a natural measure coming from volume measure in $x$ and counting measure in $P_1, P_2$.  We define incidence triples at angle $\theta$ as

$$ Tri_\theta(\PP) := \{ (x, P_1, P_2) \in Tri(\PP), \textrm{angle}(P_1,P_2) \sim \theta \}. $$

We say that $\theta$ is a typical angle of intersection for $\PP$ if $| Tri_\theta(\PP) | \approx |Tri(\PP) |$.   Assuming $\mu(\PP) \ge 2$,  we can always find a typical angle of intersection $\theta \in [a/b, 1]$.  If $\theta = a/b$, it means that the planks of $\PP$ intersect ``tangentially''.  

Next we consider slabs $S \subset B_1$ with dimensions $\theta \times 1 \times 1$.  For a given $\theta \times 1 \times 1$ slab $S$,  define

$$ \PP_S = \{ P \in \PP:  P \subset S \textrm{ and  angle}(TP, TS) \lesssim \theta \}. $$

We note that

$$\PP= \sqcup_S \PP_S, $$

\noindent where the union is over all essentially distinct $\theta \times 1 \times 1$ slabs $S \subset B_1$.   

\begin{lemma} \label{lemmaredslab} Suppose that $\theta$ is a typical intersection angle of $\PP$.  Then

$$ \mu(\PP) \approx \max_S \mu(\PP_S). $$

\end{lemma}

\begin{proof} [Proof sketch] Since $\theta$ is a typical angle of intersection, we typically see that if $P \in \PP_S$, then a good fraction of the other planks that intersect $\PP$ also lie in $\PP_S$.  
\end{proof}

Next we will take advantage of our understanding of slabs.  Lemma \ref{slab3} gives us the following local bound about the intersections of planks:

\begin{lemma} \label{lemmathicken} Suppose that $\theta$ is a typical angle of intersection for $\PP$.   Consider a typical plank $P_1$ and a typical point $x \in Y(P_1)$.  Let $Q$ be a box of dimensions $\theta b \times b \times b$,  centered at $x$,  with the long directions parallel to $T P_1$.   Then

\begin{equation} \label{eqPfillQ}
 |U(\PP) \cap Q| \gtrapprox |Q|. 
 \end{equation}

\end{lemma}

\begin{proof} [Proof sketch]  Let $\PP_Q$ be the set of $P \in \PP$ so that $P$ intersects $Q$ and the angle between $P$ and $P_1$ is $\lesssim \theta$.  For each $P \in \PP_Q$, $P\cap Q$ is a slab of dimensions $a \times b \times b$.  For typical $x, P_1$, the slabs of $\PP_Q$ have $\mu(\PP_Q) \gg 1$ and typical intersection angle $\theta$.  By Lemma \ref{slab3},  we see that

\begin{equation} \label{eqPfillQ2}
 |U(\PP) \cap Q| \gtrapprox |Q|. 
 \end{equation}

\end{proof}

Suppose $\theta$ is a typical intersection angle for $\PP$.  Choose $S$ a $\theta \times 1 \times 1$ slab so that $\mu(\PP) \approx \mu(\PP_S)$.  Next we let $\PP_{thick, S}$ denote thickened planks with dimensions $\theta b \times b \times 1$.  (If $\theta = a/b$, then $\PP_{thick,S}$ is just $\PP_S$.  Otherwise, the planks of $\PP_{thick,S}$ are thicker than those of $\PP_S$.)   We can define $\PP_{thick, S}$ for each $S$,  and we define $P_{thick} = \cup_S P_{thick, S}$.  After pigeonholing, we can assume that each plank $P_{thick} \in \PP_{thick}$ contains $N$ planks of $\PP$.    

\begin{lemma} \label{lemmaredplanktube} In the setup in the last paragraph, we have

$$ |U(\PP_S)| \approx |U(\PP_{thick, S})|,  $$

and 

$$ |U(\PP)| \approx \sum_S |U(\PP_{thick,S})|. $$

Equivalently, 

$$ \mu(\PP) \approx \mu(\PP_S)  \approx  \frac{aN }{b \theta} \mu(\PP_{thick, S}). $$

\end{lemma}

\begin{proof}   
By Lemma \ref{lemmathicken},  for a typical thick plank $P_{thick, S} \subset \PP_{thick, S}$,  we have

$$ | U(\PP_S) \cap P_{thick, S} | \gtrapprox |P_{thick, S}|. $$

So after defining an appropriate shading on each $P_{thick, S}$,  we have

$$ |U(\PP_S) | \gtrapprox |U(\PP_{thick,S})|. $$

By Lemma \ref{lemmaredslab},  $\mu(\PP) \approx \mu(\PP_S)$ and so (after some pigeonholing), 

$$ |U(\PP)| \approx \sum_S |U(\PP_S)| \approx \sum_S |U(\PP_{thick,S})|. $$

Next we can control $\mu(\PP_S)$ by relating it to $|U(\PP_S)|$.   We know that $\mu(\PP_S) = \frac{ |P|  |\PP_S|}{|U(\PP_S)|}$ and $\mu(\PP_{thick,S}) = \frac{ |P_{thick}| |\PP_{thick, S}|}{|U(\PP_{thick,S})|}$.   We have

$$ |\PP_{thick, S}| \approx \frac{1}{N} |\PP_S|, $$

$$ |P_{thick}| = \frac{ \theta b}{a} |P|.  $$

Plugging in gives 

$$  \mu(\PP_S)  \approx  \frac{aN }{b \theta} \mu(\PP_{thick, S}). $$

Finally,  by Lemma \ref{lemmaredslab},  $\mu(\PP) \approx \mu(\PP_S)$.  

 \end{proof}

By a linear change of variables, we can convert $\PP_{thick, S}$ from a set of $\theta b \times b \times 1$ planks in $S$ into a set $\TT$ of $b \times b \times 1$ tubes in $B_1$.   To see this more carefully,  recall that $\PP_{thick, S}$ consists of $\theta b \times b \times 1$ planks $P$ in the slab $S$ so that $TP$ is tangent to $TS$.   The slab $S$ has dimensions $\theta \times 1 \times 1$.   Let $L: S \rightarrow B_1$ be a linear change of variables.   Because of the tangency condition on $P$,  $L(P)$ has dimensions $b \times b \times 1$.  

To summarize,  we have reduced the original incidence problem for $\PP$ to an incidence problem for $\PP_{thick, S}$,  which is equivalent to an incidence problem for tubes $\TT$.   If we assume $K_{KT}(\beta)$ or $K_F(\beta)$, then we can get bounds for $\mu(\TT)$ and hence bounds for $\mu(\PP)$.   The final bound for $\mu(\PP)$ depends on what geometric information we have about $\PP$.   Several different scenarios will occur in the course of the paper.

\begin{lemma} \label{plankKT1}  Suppose that $K_{KT}(\beta)$ holds.   If $\PP$ is a set of $a \times b \times 1$ planks in $B_1$ so that

\begin{itemize}

\item  $ \Delta_{max}(\PP) \lessapprox 1, $

\item The planks $\PP$ have a shading with $\lambda(\PP) \gtrapprox 1$.

\end{itemize}

then $ \mu( \PP) \lessapprox |\PP|^\beta.$

\end{lemma}

\begin{proof} Suppose $\theta$ is a typical intersection angle for $\PP$.  By Lemma \ref{lemmaredplanktube} we have

\begin{equation} \label{eqPvsP'S} \mu(\PP) \approx  \frac{aN }{b \theta} \mu(\PP_{thick, S}). \end{equation}

By a linear change of variables,  we can convert $\PP_{thick, S}$ to a set of $b$-tubes $\TT$ in $B_1$.   Since convex sets are invariant with respect to linear change of variables,  we have 

$$\Delta_{max}(\TT) = \Delta_{max}(\PP_{thick, S}) \le \Delta_{max}(\PP_{thick}) \le \frac{b \theta}{a N} \Delta_{max}(\PP) \lessapprox \frac{b \theta}{a N}. $$

Now by Lemma \ref{genKKT},  we have

$$ \mu(\PP_{thick, S}) = \mu(\TT) \lessapprox \Delta_{max}(\TT)^{1 - \beta} |\TT|^\beta \lessapprox \left(  \frac{b \theta}{a N} \right)^{1 - \beta} ( | \PP_{thick, S}|)^\beta \lesssim \left(  \frac{b \theta}{a N} \right)^{1 - \beta} (N^{-1}  | \PP_S|)^\beta $$

And so plugging in (\ref{eqPvsP'S}),  we get

\begin{equation} \label{eqnearendplankKT1} \mu(\PP) \approx  \frac{aN }{b \theta} \mu(\PP_{thick, S}) \lessapprox  \left(  \frac{b \theta}{a N} \right)^{- \beta} (N^{-1} | \PP[S]|)^\beta =\left( \frac{a}{b \theta} \right)^\beta | \PP_S|^\beta.
\end{equation}

We know $a \le b \theta$ and $| \PP_S| \le | \PP|$,  and so 

$$ \mu(\PP) \lessapprox |\PP|^\beta.$$

\end{proof}

Remark.  If we examine the equality case in the previous argument,  we see that equality occurs only if $a = b \theta$ and $|\PP_S| = |\PP|$.   This means that $\theta = a/b$,  and that all the planks of $\PP$ lie in a single slab $S$ of thickness $\theta = a/b$.  If we add an assumption that the planks don't cluster into slabs,  then we get a stronger bound in the case $a \ll b$.

\begin{lemma} \label{plankKT2}  Suppose that $K_{KT}(\beta)$ holds.   If $\PP$ is a set of $a \times b \times 1$ planks in $B_1$ so that

\begin{itemize}

\item  $ \Delta_{max}(\PP) \lessapprox 1, $

\item For any $\theta \times 1 \times 1$ plank $S$,  $| \PP[S] | \lessapprox \theta | \PP|$, 

\item The planks $\PP$ have a shading with $\lambda(\PP) \gtrapprox 1$.

\end{itemize}

then

$$ \mu( \PP) \lessapprox \left( \frac{a}{b} \right)^\beta |\PP|^\beta.$$

\end{lemma}

\begin{proof} We follow the proof of Lemma \ref{plankKT1} up to (\ref{eqnearendplankKT1}),  so we have

$$ \mu(\PP) \lessapprox \left( \frac{a}{b \theta} \right)^\beta | \PP_S|^\beta.$$

Using the second bullet point,  we have $|\PP_S| \le |\PP[S]| \lesssim \theta |\PP|$.   Plugging in,  we get

$$ \mu(\PP) \lessapprox  \left( \frac{a}{b} \right)^\beta | \PP|^\beta.$$

\end{proof}

There is a similar lemma in the Frostman case.

\begin{lemma} \label{plankF}  Suppose that $K_{F}(\beta)$ holds.   Suppose $\PP$ is a set of $a \times b \times 1$ planks in $B_1$ with $C_F(\PP, B_1) \lessapprox 1$.  

Suppose that $\theta$ is a typical intersection angle for $\PP$.   Let $\PP_{thick}$ be thickened planks with dimensions $\theta b \times b \times 1$,  and suppose that each $P_{thick} \in \PP_{thick}$ contains $\sim N(\theta)$ planks of $\PP$.  Then

$$ |U(\PP)| \gtrapprox b^{2 \beta}  \left( |P_{thick}| |\PP_{thick}| \right)^{ \beta}. $$

Equivalently,

$$ \mu( \PP) \lessapprox   \left( \frac{N(\theta)}{\theta} \right)^\beta \frac{a}{b} b^{-2 \beta} \left( b^2 |\PP| \right)^{1 - \beta}. $$



\end{lemma}

\begin{proof} Since $\theta$ is a typical angle of intersection for $\PP$.   Lemma \ref{lemmaredslab} and Lemma \ref{lemmaredplanktube} tell us that, after some pigeonholing,  we can find a set $\SSS$ of $\theta \times 1 \times 1$ slabs so that

$$ |U(\PP) | \approx |\SSS| |U(\PP_S)| \approx | \SSS | |U(\PP_{thick, S})|. $$


We assumed that $\PP$ is Frostman in $B_1$.  This does not imply that $\PP_{thick, S}$ is Frostman in $S$.  So we have to work out what we know about $\PP_{thick, S}$.  Let $\PP_{thick}$ be formed by thickening the planks of $\PP$ to dimensions $\theta b \times b \times 1$.  By assumption (after pigeonholing) each plank $P_{thick} \in \PP_{thick}$ contains $\approx N$ planks of $\PP$.  

Since $\PP$ is Frostman in $B_1$,  $\PP_{thick}$ is also Frostman in $B_1$,  and so 

$$|\PP_{thick, S}| \le |\PP_{thick}[S]| \lessapprox \theta |\PP_{thick}|. $$

 The set $\PP_{thick, S}$ is not necessarily Frostman, but since $\PP_{thick}$ is Frostman,  we do know that for any convex $K \subset S$,

\begin{equation} \label{eqdensPthick}
 \Delta(\PP_{thick, S}, K) \le \Delta(\PP_{thick}, K) \lessapprox \Delta(\PP_{thick}, B_1) \sim |P| |\PP_{thick}|. 
 \end{equation}

If $|\PP_{thick, S}| \approx \theta |\PP_{thick}|$,  then we would have

$$ \Delta(\PP_{thick, S}, K) \lessapprox |P| |\PP_{thick}| \approx \frac{ |P| |\PP_{thick, S}|}{\theta} =  \frac{ |P| |\PP_{thick, S}|}{|S|} = \Delta(\PP_{thick, S}, S). $$

So if $\PP_{thick, S} \approx \theta | \PP_{thick} |$,  then $\PP_{thick,  S}$ is Frostman.   

In general,   we enlarge $\PP_{thick, S}$ to a set $\tilde \PP_{thick,S}$ with $|\tilde \PP_{thick,S}| \approx \theta |\PP_{thick}|$ by taking the union of several random translations of $\PP_{thick, S}$ within $S$.   There are now two cases.  In the main case,  a given plank would appear in only $\sim 1$ random translation of $\PP_{thick, S}$.   There is also a degenerate case where the union of the random translations includes a typical thick plank in $S$ many times.    The degenerate case is straightforward,  and we omit the details.

 In the main case, (\ref{eqdensPthick}) guarantees that $\tilde \PP_{thick, S}$ is Frostman and so all together we have
 
 \begin{itemize}
 
\item $|U(\tilde \PP_{thick,  S})| \approx \frac{ \theta | \PP_{thick}| }{ |\PP_{thick, S}|} |U(\PP_{thick, S})| \approx \theta |\SSS| |U(\PP_{thick, S})|.$

\item $C_F(\tilde \PP_{thick, S}, S) \lessapprox 1$.

\item  $|\tilde \PP_{thick, S}| \approx \theta |\PP_{thick}|$.
 
  \end{itemize}

Therefore,  we have

$$ |U(\PP)| \approx |\SSS| |U(\PP_{thick, S})| \approx \theta^{-1} |U(\tilde \PP_{thick, S})| = \frac{ |U(\tilde \PP_{thick, S})|}{|S|}. $$

Now by a linear change of variables,  we can convert $\tilde \PP_{thick,S}$ to a set of $b$-tubes $\tilde \TT$ in $B_1$.   Since convex sets are invariant with respect to linear change of variables,  $\tilde \TT$ is also Frostman.    And we have

$$ |U(\PP)| \approx  \frac{ |U(\tilde \PP_{thick, S})|}{|S|} \sim |U(\tilde \TT)|. $$

By definition of $K_F(\beta)$,  we have

$$ |U(\tilde \TT)| \gtrapprox b^{2\beta} \left( b^2 |\tilde \TT| \right)^\beta \approx b^{2 \beta} \left( b^2 \theta |\PP_{thick} | \right)^\beta = b^{2 \beta}  \left( |P_{thick}| |\PP_{thick}| \right)^{ \beta}.  $$

This gives the desired lower bound for $|U(\PP)|$. 

Next,  to get an upper bound for $\mu(\PP)$,  we relate $\mu(\PP)$ to $|U(\PP)|$.   This part is just algebra and we leave the details to the reader.

\end{proof}

\section{Locating sticky Kakeya sets} \label{seclocatesticky}

Our main goal is to reduce the general case of the Kakeya conjecture to the sticky case.   In order to do that,  we need to recognize when a set of tubes $\TT$ contains a sticky Kakeya set and when is a subset of a sticky Kakeya set.  

The proof is given by a multiscale random selection.   First we need a simple multiscale lemma about $\Delta_{max}$.

\begin{lemma} \label{lemmasubmultD}  Suppose that $\TT$ is uniform.   Suppose $\eps > 0$.   Suppose that $\delta = \rho_M \le \rho_{M-1} \le ... \le \rho_1 \le \rho_0 = 1$ is a sequence of scales.    Then

$$ \Delta_{max}(\TT) \lessapprox \prod_{m=1}^M \Delta_{max}( \TT_{\rho_m} [T_{\rho_{m-1}}] ). $$

\end{lemma}

\begin{proof}  Suppose that $K \subset B_1$ is a convex set.   We have to estimate $\Delta(\TT, K) = \frac{ |T| | \TT[K] | }{|K|}$.

Let $K_m$ denote the $\rho_m$-neighborhood of $K$.   For a tube $T \in \TT$,  let $T_{\rho_m}$ be the corresponding thickened tube in $\TT_{\rho_m}$,  and notice that if $T \subset K$ then $T_{\rho_m} \subset K_m$.    So if $T \subset K$,  then $T_{\rho_1} \subset K_1$,  and $T_{\rho_2} \subset T_{\rho_1} \cap K_2$ and in general $T_{\rho_m} \subset T_{\rho_{m-1}} \cap K_m$.   So 

\begin{equation} \label{eqtelprod}
 | \TT [K] | \lessapprox \prod_{m=1}^M | \TT_{\rho_m}[ T_{\rho_{m-1}} \cap K_m] | \lessapprox \prod_{m=1}^M \Delta_{max}( \TT_{\rho_m}[T_{\rho_{m-1}}] ) \frac{ |K_m \cap T_{\rho_{m-1}} | } {|T_{\rho_m}|}.
 \end{equation}

Now

$$  \frac{ |K_m \cap T_{\rho_{m-1}} | } {|T_{\rho_m}|} =  \frac{ |K_m \cap T_{\rho_{m-1}} | } {|T_{\rho_{m-1}}|} \cdot  \frac{ | T_{\rho_{m-1}} | } {|T_{\rho_m}|} =  \frac{ |K_m | } {|K_{m-1}|} \cdot  \frac{ | T_{\rho_{m-1}} | } {|T_{\rho_m}|}. $$

Plugging this into (\ref{eqtelprod}) and simplifying the telescoping product,  we get

$$ | \TT [K] |  \lessapprox  \left[ \prod_{m=1}^M \Delta_{max}( \TT_{\rho_m}[T_{\rho_{m-1}}] ) \right] \frac{ |K| } {|T|}.$$

\end{proof}

\begin{lemma} \label{lemmadecsticky} Suppose that $\TT$ is uniform.   Suppose $\eps > 0$.   Suppose that $\delta = \rho_M \le \rho_{M-1} \le ... \le \rho_1 \le \rho_0 = 1$ is a sequence of scales.   Suppose that for each $m$,

\begin{itemize}

\item $\frac{\rho_{m}}{\rho_{m-1}} \gtrapprox \delta^{\eps}$ for all $m$.   

\item $C_F(\TT_{\rho_m}[T_{\rho_{m-1}}]) \lessapprox \left( \frac{ \rho_{m-1} }{\rho_{m}} \right)^\eps$.

\end{itemize}

Then $\TT$ contains a subset $\TT'$ which is $\delta^{-3\eps}$-sticky and has $\Delta_{max}(\TT') \lessapprox \delta^{-3 \eps}$.

\end{lemma}

\begin{proof} We select $\TT'$ in a random way one level of the tree at a time.   We first select $\TT'_{\rho_1} \subset \TT_{\rho_1}$ randomly with cardinality roughly $ \rho_1^{-2}$.     Then for each $T_{\rho_1} \in \TT'_{\rho_1}$,  we select $\TT'_{\rho_2}[T_{\rho_1}]$ randomly with cardinality roughly $(\rho_1/ \rho_2)^{-2}$.   Proceeding in this way defines $\TT'$.   We have to check some quantitative bounds for the sets of tubes that appear in this process.

Because of the second bullet point,  we know that for each $m$

$$ | \TT_{\rho_m}[ T_{\rho_{m-1}} ] |  \gtrapprox \left( \frac{\rho_{m-1}}{\rho_{m}} \right)^{-2 + \eps}. $$

So for each $T_{\rho_{m-1}} \in \TT'_{\rho_{m-1}}$  we can choose $\TT'_{\rho_m}[T_{\rho_{m-1}}] \subset \TT_{\rho_m} [T_{\rho_{m-1}}]$ randomly with cardinality in the range

$$  \left( \frac{\rho_{m-1}}{\rho_{m}} \right)^{-2 + \eps} \lessapprox  | \TT'_{\rho_m}[T_{\rho_{m-1}}] | \lessapprox  \left( \frac{\rho_{m-1}}{\rho_{m}} \right)^{-2 } . $$

This defines $\TT'$,  and it implies that for every $\rho_m$,  

$$ \delta^{\eps} (\rho_m /\delta)^2 \lessapprox | \TT'[T_{\rho_m}] | \lessapprox (\rho_m / \delta)^2.  $$

Now by the first bullet point,  we see that for any $\rho$ 

$$ \delta^{3 \eps} (\rho /\delta)^2 \lessapprox | \TT'[T_{\rho_m}] | \lessapprox \delta^{-3 \eps} (\rho / \delta)^2.  $$

In other words,  $\TT'$ is $\delta^{- 3 \eps}$-sticky.

By a randomization argument,  $\TT'_{\rho_m}[T_{\rho_{m-1}}] $ obeys

$$C_F(\TT'_{\rho_m}[T_{\rho_{m-1}}]) \lessapprox \left( \frac{ \rho_{m-1} }{\rho_{m}} \right)^\eps$$

Given our bound for $| \TT'_{\rho_m}[T_{\rho_{m-1}}] |$,  this implies

$$\Delta_{max} (\TT'_{\rho_m}[T_{\rho_{m-1}}]) \lessapprox \left( \frac{ \rho_{m-1} }{\rho_{m}} \right)^\eps.$$

By the submultiplicative property of $\Delta_{max}$ in Lemma \ref{lemmasubmultD}, 
we get

$$ \Delta_{max}(\TT') \lessapprox \delta^{- \eps}. $$

\end{proof}

There is an analogous lemma to recognize when $\TT$ is contained in a sticky Kakeya set.

\begin{lemma} \label{lemmasubsticky} Suppose that $\TT$ is uniform.   Suppose $\eps > 0$.   Suppose that $\delta = \rho_M \le \rho_{M-1} \le ... \le \rho_1 \le \rho_0 = 1$ is a sequence of scales.   Suppose that for each $m$,

\begin{itemize}

\item $\frac{\rho_{m}}{\rho_{m-1}} \gtrapprox \delta^{\eps}$ for all $m$.   

\item $\Delta_{max}(\TT_{\rho_m}[T_{\rho_{m-1}}]) \lessapprox \left( \frac{ \rho_{m-1} }{\rho_{m}} \right)^\eps$.

\end{itemize}

Then $\TT$ is contained in a subset $\TT'$ which is $\delta^{-3\eps}$-sticky and has $\Delta_{max}(\TT') \lessapprox \delta^{-3 \eps}$.

\end{lemma}

The proof is similar and we leave the details for the reader.

\section{Main Lemma 1} \label{secmainlemma1}

In this section, we sketch the proof of Main Lemma \ref{lemmain1}.  We first recall the statement.  

\begin{mlem*}  $K_{KT}(\beta)$ implies $K_F(\beta)$.  
\end{mlem*}

Let's first digest what the lemma is saying.   Suppose that $\TT$ is a set of $\delta$ tubes in $B_1$ with $C_F(\TT) \lessapprox 1$.   We have to prove the bound

\begin{equation} \label{ml1goal}
 \mu( \TT) \lessapprox ( \delta^{-2})^\beta \left( \delta^2 | \TT | \right)^{1 - \beta}. 
\end{equation}

\noindent Since $C_F(\TT, B_1) \lessapprox 1$,  we have $|\TT| \gtrapprox \delta^{-2}$.   If $|\TT| \approx \delta^{-2}$,  then $C_F(\TT, B_1) \lessapprox 1$ is equivalent to $\Delta_{max}(\TT) \lessapprox 1$,  and then $K_{KT}(\beta)$ gives $\mu(\TT) \lessapprox \delta^{-2 \beta}$,  which is the desired bound.   So the case $| \TT | \approx \delta^{-2}$ is trivial.   So the content of the lemma is in the case when $ | \TT | \gg \delta^{-2}$.

If $| \TT | \gg \delta^{-2}$,  then $\Delta_{max}(\TT) \gg 1$, so we can not immediately apply $K_{KT}(\beta)$ to $\TT$.    If $\tilde \TT \subset \TT$ is a random subset with $| \tilde \TT | \sim \delta^{-2}$,  then it's not hard to check that $\Delta_{max}(\tilde \TT) \lessapprox 1$.   We can apply $K_{KT}(\beta)$ to $\tilde \TT$,  and we get $\mu(\tilde \TT) \lessapprox \delta^{-2 \beta}$ which is equivalent to $|U(\tilde \TT)| \gtrapprox \delta^{2 \beta}$.   We have $|U(\TT)| \ge |U(\tilde \TT)|$,  and this translates into

\begin{equation} \label{ml1triv}
 \mu( \TT) \lessapprox ( \delta^{-2})^\beta \left( \delta^2 | \TT | \right)
\end{equation}

Comparing (\ref{ml1triv}) with our goal (\ref{ml1goal}),  we see that we have reduced the exponent on the last factor from 1 to $1 - \beta$.   In the proof of Kakeya,  any exponent strictly less than 1 will work.   We can also describe the lemma in terms of $|U(\TT)|$,  which may be more intuitive.   Since we assume $K_{KT}(\beta)$,  we already know that if $C_F(\TT) \lessapprox 1$ and $| \TT | \sim \delta^{-2}$,  then $|U(\TT)| \gtrapprox \delta^{2 \beta}$.   Now we need to show that if $C_F(\TT) \lessapprox 1$ and if $ | \TT | $ is far larger than $\delta^{-2}$,  then $|U(\TT)|$ is a little larger than $\delta^{2 \beta}$.   While this may sound intuitive,  it is not at all trivial to prove, and the proof will depend crucially on sticky Kakeya.

We organize the proof as a bootstrapping or self-improving argument.

\begin{lemma} \label{lemmain1boot} Suppose that $K_{KT}(\beta)$ holds and that $K_F(\gamma)$ holds for some $\gamma > \beta$.  Then $K_F(\gamma - \nu)$ holds for $\nu = \nu(\gamma, \beta) > 0$.
\end{lemma}

Since $K_F(1)$ holds trivially, Lemma \ref{lemmain1boot} implies Main Lemma \ref{lemmain1}.  

The detailed proof of Lemma \ref{lemmain1boot} involves some delicate keeping track of various small parameters.  Including these small parameters is a little messy,  but it is important.   In this section,  we sketch the proof without keeping track of everything in detail.   We will not be precise about the exact meaning of the symbols $\lessapprox$ and $\ll$.  

We will choose $\nu(\beta, \gamma) > 0$ below, very small compared to $\exsticky$ from the sticky Kakeya theorem.  

Suppose that $\TT$ is a set of $\delta$-tubes in $B_1 \subset \RR^3$ with $C_F(\TT) \lessapprox 1$.  
 We have to show that 

\begin{equation} \label{mugoalml1}
 \mu( \TT) \lessapprox   \delta^{\nu} ( \delta^{-2})^\gamma \left( \delta^2 | \TT | \right)^{1 - \gamma}. 
 \end{equation}

We can assume without loss of generality that $\TT$ is uniform.

Starting with $\TT$,  we say that $\rho$ is a sticky intermediate scale if $\delta \ll \rho \ll 1$ and $C_F (\TT[T_\rho])$ is almost 1.   (We use the word sticky here because we will see below that if every intermediate scale is sticky,  then  $\TT$ contains a sticky Kakeya set $\TT'$. )

If $\TT$ has a sticky intermediate scale $\rho$,  then we can bound $\mu(\TT)$ by

$$ \mu(\TT) \le \mu(\TT [T_\rho]) \mu(\TT_\rho), $$

\noindent reducing to two smaller problems.  After rescaling, we can think of $\TT [T_\rho]$ as a set of $(\delta/\rho)$-tubes in $B_1$.  We let $\TT_1$ denote this set of tubes and $\TT_2 = \TT_\rho$, so that $\mu(\TT) \lessapprox \mu(\TT_1) \mu(\TT_2)$.  Now we can apply the same process to $\TT_1$ and $\TT_2$.    

We use this process as part of a stopping time argument.  The stopping criterion is as follows.   We choose a small parameter $\eps$,  setting $\eps = \exsticky / 10$.   Suppose that $\tilde \TT$ is a set of $\tilde \delta$-tubes in $B_1$.    If $\tilde \delta \ge \delta^\eps$, we stop.   If $\tilde \TT$ has a sticky intermediate scale $\tilde \rho$, then we bound $\mu(\tilde \TT) \le \mu(\tilde \TT[T_{\tilde \rho}]) \mu(\tilde \TT_{\tilde \rho})$.   If $\tilde \TT$ does not have a good intermediate scale, then we stop.  The stopping time argument outputs a bound of the form

$$ \mu(\TT) \le \prod_j \mu(\TT_j), $$

\noindent where $\TT_j$ is a set of $\delta_j$-tubes in $B_1$ with the following properties:

\begin{itemize}

\item For each $j$,  $C_F(\TT_j)$ is small.

\item For each $j$, either $\delta_j \ge \delta^{\eps}$ or $\TT_j$ has no sticky intermediate scale.  

\end{itemize}

If every $\delta_j$ obeys $\delta_j \ge \delta^\eps$,  then we can use Lemma \ref{lemmadecsticky} to show that $\TT$ contains a sticky Kakeya set $\TT'$.   Unwinding the stopping time argument,  we see that in this case there is a sequence of $\delta = \rho_J \le ... \le \rho_1 \le \rho_0 = 1$ so that $\TT_j$ corresponds to $\TT_{\rho_j}[T_{\rho_{j-1}}]$.   By our stopping criteria,  we see that $C_F(\TT_{\rho_j}[T_{\rho_{j-1}} ],  T_{\rho_{j-1}})$ is roughly 1 and that each ratio $\rho_j / \rho_{j-1}$ is at least $\delta^\eps$.   Then Lemma \ref{lemmadecsticky} tells us that $\TT$ contains a subset $\TT'$ which is $\delta^{- \exsticky}$-sticky and obeys $\Delta_{max}(\TT') \lessapprox \delta^{-\exsticky}$.   Then the sticky Kakeya theorem,  Theorem \ref{thmstickykak},  tells us that $\mu(\TT')$ is roughly 1 and so $| U(\TT') |$ is roughly 1.   Then $|U(\TT)| \ge |U(\TT')|$, and so we get essentially sharp bounds for $|U(\TT)|$ and hence also for $\mu(\TT)$.  

Otherwise, there is at least one $j$ so that $\TT_j$ has no sticky intermediate scale.   Before we deal with this scale,  let us discuss how $\mu(\TT)$ relates to $\mu(\TT_j)$.   Recall that we have $\mu(\TT) \lessapprox \prod_j \mu(\TT_j)$.   
We can apply $K_F(\gamma)$ to each $\TT_j$,  and this would recover the same bound that we get by applying $K_F(\gamma)$ to $\mu(\TT)$.   In order to beat this bound and prove (\ref{mugoalml1}),  it suffices to beat the bound for a single $\TT_j$.   Now we focus on the specific $\TT_j$ that has no sticky intermediate scale,  and we prove a bound for $\mu(\TT_j)$ that beats the bound from $K_F(\gamma)$.

Recall that $\TT_j$ that has no sticky intermediate scale and that $\delta_j \le \delta^{\exsticky}$.   Pick $\sigma = \delta_j^{1/2}$.   Since $\TT_j$ has no sticky intermediate scale,  $C_F(\TT_j[T_\sigma]) \gg 1$.   Let $\WW_{T_\sigma}$ be the maximal density factoring of $\TT_j[T_\sigma]$ given by Lemma \ref{lemmafactmax}.   Suppose that each $W \in \WW_{T_\sigma}$ has dimensionsions $a \times b \times 1$ with $ a\le b \le \sigma < 1$.    The maximal density factoring lemma tells us that $C_F(\TT_j[W],  W) \lessapprox 1$. 

Notice that if $a=b$,  then $W = T_b$,  and so $C_F(\TT_j[T_b]) \lessapprox 1$.   Since $\TT_j$ has no sticky intermediate scale,  this is only possible if $b$ is very close to $\delta$.   If we quantify this argument,  we see that either $a \ll b$ (the eccentric case) or $b \lessapprox \delta$ (the small case).

\vskip10pt

{\bf The eccentric case}

\vskip10pt

In the eccentric case,  we set $\WW = \sqcup_{T_\sigma \in \TT_\sigma} \WW_{T_\sigma}$.   Then we can bound $\mu(\TT_j) \lessapprox \mu(\TT_j[W]) \mu(\WW)$.   Now we bound each of these factors using Lemma \ref{plankF}. 

To use Lemma \ref{plankF}, we need to check that each set involved is Frostman.  
Notice that $C_F(\TT_j[W], W) \lessapprox 1$ by the factoring lemma,  and $C_F(\WW) \lessapprox C_F(\TT) \lessapprox 1$.  

First we estimate $\mu(\WW)$.    To use Lemma \ref{plankF},  we will need to estimate $N(\theta)$,  the number of $a \times b \times 1$ planks in a $\theta b \times b \times 1$ thick plank.
For $\WW$,  we can estimate $N(\theta)$ by noting that $\Delta_{max} (\WW_{T_\sigma}) \lessapprox 1$.  A given thick plank $W'$ is contained in a tube $T_\sigma$, and we have $\WW[W'] = \WW_{T_\sigma}[W']$ and so $N(\theta) |P| \lessapprox |P'|$.  Therefore, for every $\theta \in [a/b, 1]$, we have
 
 $$N(\theta) \lessapprox \frac{\theta b}{a}. $$
 
 \noindent Plugging into Lemma \ref{plankF}, we get
 
 \begin{equation} \label{eqmuW}  \mu(\WW) \lessapprox  (a/b)^{1 - \gamma} b^{-2 \gamma} (b^2 |\WW|)^{1 - \gamma} \end{equation}
  
 Next we estimate $\mu(\TT_j[W])$.  For $\TT_j[W]$,  we first have to make a change of coordinates that maps $W$ to $B_1$ and each $\delta_j  \times \delta_j  \times 1$ tube  $T \in \TT_j[W]$ to a $b^{-1} \delta_j \times a^{-1} \delta_j \times 1$ plank.   Next we need to bound $N(\theta)$,  the number of such planks in a $\theta a^{-1} \delta_j \times a^{-1} \delta_j \times 1$ thick plank.   Undoing the coordinate change, $N(\theta)$ is the number of tubes of $\TT$ in a $\delta_j \times \frac{\theta b}{a} \delta_j \times 1$ plank.  Since the tubes of $\TT$ are essentially distinct, we get $ N(\theta) \lessapprox \left( \frac{\theta b}{a} \right)^2$ and so
 
 $$ N(\theta) / \theta \lessapprox (b/a)^2 \theta \le (b/a)^2. $$ 
 
 Now we plug into Lemma \ref{plankF}, remembering that the dimension of each plank is $b^{-1} \delta_j \times a^{-1} \delta_j \times 1$.  Using our bound for $N(\theta)$, we get
  
 \begin{equation} \label{eqmuTW}  \mu(\TT_j[W]) \lessapprox (a/b)^{1 - 2 \gamma} (a^{-1} \delta_j)^{-2 \gamma} (a^{-2} \delta_j^2 | \TT_j[W]| )^{1 - \gamma}. \end{equation}
 
 We have $\mu(\TT_j) \lessapprox \mu(\WW) \mu(\TT_j[W])$.  Bounding the factors using  (\ref{eqmuW}) and (\ref{eqmuTW}) we get
 
 $$ \mu(\TT_j) \lessapprox \left( \frac{a}{b} \right)^\gamma \delta_j^{-2 \gamma} \left( \frac{\TT}{\delta_j^{-2}} \right)^{1 - \gamma},$$

\noindent which is a factor of $\left( \frac{a}{b} \right)^\gamma $ better than the bound from $K_F(\gamma)$.    In the eccentric case, $a/b \ll 1$ and so we improve significantly on the bound from $K_F(\gamma)$, which gives the desired conclusion (\ref{mugoalml1}).

\vskip10pt

{\bf The small case}

\vskip10pt

In the small case,  $\WW$ is essentially equal to $\TT_j$,  and so $\Delta_{max}(\TT_j) \approx \Delta_{max}(\WW) \lessapprox 1$.   Therefore,  we can bound $\mu(\TT_j)$ using $K_{KT}(\beta)$.   We note that since $C_F(\TT_j) \lessapprox 1$ and $\Delta_{max}(\TT_j) \lessapprox 1$,  we must have $|\TT_j| \approx \delta_j^{-2}$.   Now $K_{KT}(\beta)$ gives

$$ \mu(\TT_j) \lessapprox | \TT_j|^{\beta} \approx \delta_j^{- 2 \beta} ( \delta_j^2 | \TT_j | )^{1 - \beta}. $$

Since $\beta < \gamma$,  this bound improves on the one from $K_F(\gamma)$ and gives the desired conclusion (\ref{mugoalml1}).

\section{Main Lemma 2 overview}   \label{secmainlemma2over}

Our last goal is to prove Main Lemma \ref{lemmain2}.  We recall the statement.  

\begin{mlem*}  If $K_{KT}(\beta)$ and $K_{F}(\beta)$ hold,  then $K_{KT}(\beta - \nu)$ holds for $\nu = \nu(\beta) > 0$.   
\end{mlem*}

\subsection{Initial reductions}

The sticky Kakeya theorem plays a crucial role in this proof.   By using sticky Kakeya, we are able to reduce the problem to a situation which is not sticky in a strong sense. 

\begin{lemma} \label{lemmain2vns} (Very not sticky case of Main Lemma \ref{lemmain2}, approximate version) Suppose $\exscal > 0$.   Suppose that $K_{F}(\beta)$ holds for some $\beta > 0$.   Suppose that $\TT$ is a set of $\delta$-tubes in $B_1$ obeying

\begin{itemize}

\item $\Delta_{max}(\TT) \lessapprox 1$.

\item For each $\rho \in [\delta^{1 - \exscal}, \delta^{\exscal}]$,  $| \TT_\rho | \gg \rho^{-2}$.  

\end{itemize}

Then $\mu(\TT) \lessapprox |\TT|^{\beta - \nu}$,  where $\nu = \nu(\beta, \exscal) > 0$.

\end{lemma}

The precise statement of Lemma \ref{lemmain2vns} has a precise version of $\ll$,  which involves choosing some small exponents.

The proof that Lemma \ref{lemmain2vns} implies Main Lemma \ref{lemmain2} is similar to the proof of Main Lemma \ref{lemmain1},  using Lemma \ref{lemmasubsticky} instead of Lemma \ref{lemmadecsticky}.  We omit the details.

\subsection{Looking at smaller balls $B \subset B_1$}

As in our first glimmer of the proof,  we will study how $\TT$ behaves in smaller balls.   We pick the scale  $\rho = \delta^{1 - \exscal}$ and let $r = \delta/\rho = \delta^{\exscal}$.    Next let $B = B(x,  r)$ be a typical ball in $U(\TT_{r})$. 

As in (\ref{multprod}) in Section \ref{secglimmer},  we can define $\TT_B$,  a set of $\delta \times \delta \times r$ tubes in $B$ so that

\begin{equation} \label{multprod}
 \mu(\TT) \lessapprox \mu(\TT[T_\rho]) \mu(\TT_B). 
 \end{equation}
 
 We apply Lemma \ref{lemmafactmax} to $\TT_B$,  and we let $\WW$ be the maximal density factoring of $\TT_B$.   By pigeonholing,  we can suppose that each $W \in \WW$ has dimensions $a \times b \times r$,  with $a \le b \le r$.
 
 The way we proceed depends on the shape of $W$.   If $a \gg \delta$,  then we are in the thick case.   If $a \lessapprox \delta$,  then we are in the thin case.    We deal with these cases in the next two sections.

\section{The thick case}

 In the thick case,  we will show that $U(\TT)$ fills a large fraction of a typical ball $B_a \subset U(\TT_a)$,  and we will use this to show that $U(\TT)$ is large.

\subsection{Density of $U(\TT)$ in small balls} \label{subsecdenssmallballs}

Recall that $\lambda(\TT) \gtrapprox 1$, and so double counting gives us

$$ \mu(\TT) | U(\TT) | \approx |\TT| |T| .  $$

\noindent We are interested in upper bounds for $\mu(\TT)$, which are equivalent to lower bounds for $| U(\TT)|$.  Our goal bound $\mu(\TT) \lessapprox |\TT|^{\beta - \nu}$ is equivalent to

\begin{equation} \label{eqgoalUT}
 |U(\TT)| \gtrapprox |\TT|^{1 - \beta + \nu} \delta^2. \end{equation}

Suppose that $\delta \le r \le 1$.  We can lower bound

$$ | U(\TT) | \gtrapprox | U(\TT_r) | \frac{ |U(\TT) \cap B_r |}{|B_r|}, $$

\noindent where the fraction on the right applies to a typical ball $B_r$ in $U(\TT_r) $.  

We know $\Delta_{max}(\TT) \lessapprox 1$,  but this does not tell us $\Delta_{max}(\TT_r) \lessapprox 1$.   It does tell us that

$$ \Delta_{max}(\TT_r) \lessapprox \frac{ (r/\delta)^2 }{ | \TT[T_r] |}. $$

We choose a random subset $\TT'_r \subset \TT_r$ where each tube is included with probability $p = \frac{ |\TT[T_r] |}{ (r/\delta)^2 }$.  (Note that since $\Delta_{max}(\TT) \lessapprox 1$,  $| \TT[T_r] | \lessapprox (r / \delta)^2$ and so $p \lessapprox 1$.  If $p \ge 1$, we just set $p=1$.)

 Since $p \Delta_{max}(\TT_r) \lessapprox 1$,  we see that $\Delta_{max}(\TT'_r) \lessapprox 1$.   The cardinality of $\TT'_r$ works out in a clean way:

$$ | \TT'_r | \approx \frac{ \TT[T_r]}{ (r/\delta)^2 } | \TT_r | \approx  \frac{ \TT[T_r]}{ (r/\delta)^2 } \frac{ | \TT |}{|\TT[T_r]|} = \left( \frac{\delta}{r} \right)^2 | \TT |. $$

 Now using $K_{KT}(\beta)$, we see that

$$ | U(\TT_r)| \ge |U(\TT'_r)| \gtrapprox | \TT'_r|^{1 - \beta} r^2 \approx |\TT|^{1 - \beta} \delta^2 \left( \frac{ r}{\delta} \right)^{2 \beta}. $$

Therefore, our goal (\ref{eqgoalUT}) would follow if we knew that the following inequality holds for some $r \in [\delta, 1]$:

\begin{equation} \label{eqgoaldens}
 \frac{ |U(\TT) \cap B_r |}{|B_r|} \gtrapprox \delta^{- \nu} \left( \frac{\delta}{r} \right)^{2 \beta}
\end{equation}

In the thick case,  when $a \gg \delta$,  we will show (\ref{eqgoaldens}) for $r=a$.

\subsection{Maximal density factoring revisited}

In Lemma \ref{lemmafactmax},  we chose sets $W$ to maximize $\Delta(\VV,  W)$.   To get a good bound in the thick case,  it is important to choose carefully in the case of a tie or a near tie.   When the maximum density is achieved by two convex sets with very different sizes,  we choose $W$ to be the convex set of smaller volume.   More precisely,  we adjust our maximization process as follows.   We let $\exfact > 0$ be a tiny exponent to be chosen later.   Then,  instead of choosing $W$ to maximize $\Delta(\VV, W)$,  we choose $W$ to maximize

$$ |W|^{- \exfact} \Delta(\VV, W). $$

This leads to a variation of Lemma \ref{lemmafactmax} with two small changes.   If $K \subset W \in \WW$,  then we have a slightly stronger estimate

\begin{equation} \label{factmaxmod1}
\Delta(\VV_{uni}[W], K) \lessapprox \left( \frac{ |K|}{|W|} \right)^{\exfact} \Delta(\VV_{uni}[W], W).
\end{equation}

On the other hand,  we have a slightly weaker estimate for $\Delta_{max}(\WW)$.   If the sets of $\VV$ are all contained in a large convex set $\Omega$,  then 

\begin{equation} \label{factmaxmod2}
\Delta_{max}(\WW) \lessapprox \left( \frac{|\Omega|}{|W|} \right)^{\exfact}
\end{equation}

When we apply Lemma \ref{lemmafactmax} to $\TT_B$,  we actually use this small modification.   The equation (\ref{factmaxmod1}) tells us in particular that for any convex set $K \subset B$,  

\begin{equation} \label{factmaxmodbias}
\Delta(\TT_B,  W) / \Delta(\TT_B,  K) \gtrapprox \left( \frac{|W|}{|K|} \right)^{\exfact}.
\end{equation}

In particular,  taking $K = T_B$,  we have $\Delta(\TT_B,  W) \gg \Delta(\TT_B,  T_B) = 1$,  and so

\begin{equation} \label{TBWbig}
 | \TT_B[W] | \gg \frac{ |W| }{|T_B|}. 
 \end{equation}










\subsection{A simple case}

The simplest scenario in the thick case is when $W = B_r$.   The strong estimate (\ref{factmaxmod1}) implies that $C_F(\TT_B,  B) \lessapprox 1$,  and so we can apply $K_F(\beta)$ to $\TT_B$.   Strictly speaking,  to do the application we rescale $B = B(x,r)$ to the unit ball $B_1$,  which sends $\TT_B$ to a set $\tilde \TT_B$ of $\delta/r$ tubes.   Then we get

$$ \frac{  |U(\TT_B)| }{|B|} = | U(\tilde \TT_B) | \gtrapprox (\delta/r)^{2 \beta} \left( |\tilde T_B| |\tilde \TT_B| \right)^\beta. $$

By (\ref{TBWbig}), we know that

$$ | \TT_B | = |\TT_B[W]| \gg \frac{ |B| }{|T_B|} = \frac{ |B_1| }{|\tilde T_B|}, $$

and so the factor in parentheses on the right-hand side is $\gg 1$.   All together we have

$$ \frac{ |U(\TT) \cap B_r|} {|B_r|} \gtrapprox  \frac{  |U(\TT_B)| }{|B|} \gg (\delta/r)^{2 \beta}. $$

This gives (\ref{eqgoaldens}).

\subsection{The general case}

In the general version of the thick case,  $W$ has dimensions $a \times b \times r$,  with $\delta \ll a \le b \le r$.   
We need to prove that $U(\TT)$ fills up a large portion of a typical ball $B_a \subset U(\TT_a)$.   In particular,  we need to prove (\ref{eqgoaldens}):

\begin{equation} \label{eqgoaldensthick}
 \frac{ |U(\TT) \cap B_a |}{|B_a|} \gg \left( \frac{\delta}{a} \right)^{2 \beta}
\end{equation}

For a typical $B_a$ in a typical $W$,  we have

$$ \frac{ |U(\TT) \cap B_a |}{|B_a|} \gtrapprox \frac{ |U(\TT_B[W])| }{|W|}. $$

Now we make a linear change of variables to convert $W$ to $B_1$ and converting $\TT_B[W]$ to a set of planks $\PP$ in $B_1$.   We have $\frac{ |U(\TT_B[W])| }{|W|} \sim |U(\PP)|$,  so it suffices to prove

$$|U(\PP)| \gg \left( \frac{\delta}{a} \right)^{2 \beta}$$

If $W$ has dimensions $a \times b \times r$,  then a plank $P \in \PP$ has dimensions $\frac{\delta}{b} \times \frac{\delta}{a} \times 1$.  
Note that $C_F(\PP, B_1) = C_F(\TT_B[W], W) \lessapprox 1$.

Now we apply Lemma \ref{plankF}.   Let $\theta$ be a typical angle of intersection for the planks of $\PP$.   Recall that $\PP_{thick}$ are thickened planks of $\PP$ with dimensions $\theta \frac{\delta}{a} \times \frac{\delta}{a} \times 1$.    Using $K_F(\beta)$,  Lemma \ref{plankF} gives us the lower bound

$$ |U(\PP)| \gtrapprox \left( \frac{ \delta}{a} \right)^{2 \beta} \left( |P_{thick} | | \PP_{thick} | \right)^\beta. $$

So it suffices to see that 

$$|P_{thick}| |\PP_{thick} | \gg 1.$$    

This will follow from the way we selected $W$ using the modified maximal density factoring lemma.   In order to see this,  we notice that $|\PP_{thick}| = \frac{ | \PP |}{ | \PP[P_{thick}] | }$,  and so our desired inequality is equivalent to

\begin{equation} \label{goalfrac1} \frac{ |\PP | } { |B_1| } \gg \frac{| \PP[P_{thick}] |  }{|P_{thick}|}. 
\end{equation}

Recall that $\PP$ is related to $\TT_B[W]$ by a linear change of variables.   When we undo that change of variables,  the planks $\PP$ become $\TT_B[W]$ and each plank $P$ becomes a $\delta \times \delta \times r$ tube.    Each thick plank $P_{thick}$ becomes a rectangular solid $T_{thick}$ with dimensions $\delta \times b' \times r$ with $\delta \le b' \le r$.   And the whole equation (\ref{goalfrac1}) becomes

$$ \frac{ |\TT_B[W]|} {|W|} \gg \frac{|\TT_B[T_{thick}] } {|T_{thick}|},$$

or equivalently

$$ \Delta(\TT_B,  W) \gg \Delta(\TT_B, T_{thick}). $$

This will follow from our modified maximal density factoring lemma (and it is exactly the reason for the modification).   By (\ref{factmaxmodbias}),  we have

$$ \Delta(\TT_B, W) / \Delta(\TT_B, T_{thick}) \gtrapprox \left( \frac{|W|}{|T_{thick}|} \right)^{\exfact}. $$

Since $W$ has dimensions $ a\times b \times r$  and $T_{thick}$ has dimensions $\delta \times b' \times r$ with $b' \le b$,  we have

$$ \left( \frac{|W|}{|T_{thick}|} \right)^{\exfact} \ge \left( \frac{a}{\delta} \right)^{\exfact}. $$

Since we are in the thick case,  we know that $a/\delta \gg 1$,  and so we get the desired bound $ \Delta(\TT_B, W) \gg \Delta(\TT_B, T_{thick})$.  
This finishes the proof in the thick case.

\section{The thin case}

Now we consider the thin case when $a \approx \delta$.   Recall that $W$ has dimensions $a \times b \times r$. 

Since $a \approx \delta$,  when we study $\TT_B[W]$,  we are essentially reduced to a 2-dimensional Kakeya problem.   Since $C_F(\TT_B[W], W) \lessapprox 1$,  we conclude that 

\begin{equation} \label{2dkak}
|U(\TT_B[W])| \approx |W|. 
\end{equation}

\subsection{Slab case}
Recall that $r = \delta^{\exscal}$, where $\exscal > 0$ is the constant from Lemma \ref{lemmain2vns}, which is a tiny constant that we get to choose.  
 If $b \ge \delta^{\exscal} r = \delta^{2 \exscal}$,  then $W$ is essentially a giant slab of dimensions $\approx \delta \times \delta^{\exscal} \times \delta^{2 \exscal}$.   Because we can choose $\exscal  > 0$ as small as we like,  and because we have a good understanding of slabs,  this is an easy case.     So from now on,  we can assume that $b \le \delta^{\exscal} r$.  

\subsection{Transverse case}
Next we analyze the typical intersection angle of planks in $\WW$.   Let $\theta \ge a/b \approx \delta/ b$ be a typical intersection angle.   If $\theta \gg \delta/b$,  then Lemma \ref{lemmathicken} tells us that
$|U(\WW) \cap B_{\theta b} | \approx |B_{\theta b}|$.   Since $U(\TT_B[W])$ essentially fills $W$,  it follows that $|U(\TT) \cap B_{\theta b} | \approx |B_{\theta b}|$.   This gives (\ref{eqgoaldens}) with radius $r = \theta b \gg \delta$.   So from now on,  we can assume that $\theta \approx a/b \approx \delta/b$.  

\subsection{Tangential case}
Define $\TT_W = \{ T \in \TT: |T \cap W| \sim \delta^2 r \}$.   The angle between any two tubes in $\TT_W$ is $\lessapprox b/r$.  Set $\rho_2 = b/r$.   Then there is a tube $T_{\rho_2} \in \TT_{\rho_2}$ so that $\TT_W \subset \TT[T_{\rho_2}]$,  and therefore $\mu(\TT_W) \lessapprox \mu(\TT[T_{\rho_2}])$.   Every $T \in \TT$ that enters $B$ is contained in some $T_B$ which is contained in some $W$.   Using (\ref{2dkak}), we get the inequality

$$ \mu(\TT) \lessapprox \mu(\TT[T_{\rho_2}]) \mu(\WW). $$

Recall that $\theta \approx \delta/b$ is a typical angle of intersection for $\WW$.   Let $S$ denote a $\theta \times 1 \times 1$ slab.  Then by Lemma \ref{lemmaredslab},  we have $\mu(\WW) \approx \mu(\WW_S)$.  We can make a linear change of variables that converts $S$ to $B_1$ and converts $\WW$ to a set $\tilde \TT$ of $\rho_2$-tubes in $B_1$.    By the maximal density factoring lemma, we know that $\Delta_{max}(\WW)$ is small.  With the original lemma, we would have $\Delta_{max}(\WW) \lessapprox 1$.  With our modified lemma, we have $\Delta_{max}(\WW) \lessapprox \delta^{-\exfact}$.  When one carefully keeps track of all small parameters, this turns out to be essentially as good as $\Delta_{\max}(\WW) \lessapprox 1$, so in these notes we just write $\Delta_{max}(\WW)\lessapprox 1$.  Since $\Delta_{max}(\WW) \lessapprox 1$,  we have $\Delta_{max}(\WW_S) \lessapprox 1$ and so $\Delta_{max}(\tilde \TT) \lessapprox 1$.   Since $\tilde \TT$ is a set of $\rho_2$-tubes in $B_1$ with $\Delta_{max}(\tilde \TT) \lessapprox 1$, we have $| \tilde \TT | \lessapprox \rho_2^{-2}$.   Now we can bound $\mu(\WW)$ using $K_{KT}(\beta)$ by 

$$ \mu(\WW) \approx \mu(\WW_S) = \mu(\tilde \TT) \lessapprox |\tilde \TT|^\beta \lessapprox (\rho_2^{-2})^\beta \ll | \TT_{\rho_2} |^\beta.$$

On the other hand,  $\TT[T_{\rho_2}]$ is Katz-Tao,  and so $K_{KT}(\beta)$ gives 

$$ \mu(\TT[T_{\rho_2}]) \lessapprox | \TT[T_{\rho_2}] |^\beta. $$

So all together we have

$$  \mu(\TT) \lessapprox \mu(\TT[T_{\rho_2}]) \mu(\WW) \ll | \TT[T_{\rho_2}] |^\beta | \TT_{\rho_2} |^\beta = | \TT |^\beta. $$

This gives the desired bound in the thin case.

\section{What's missing?} \label{secleftout}

This is a proof outline and not a proof.   We have left out some technical details,  and in this last section we briefly summarize what needs to be filled in.

\subsection{Keeping track of small parameters}

In these notes we are not completely precise about the meaning of $\lessapprox$ and $\ll$.   An inequality that we wrote as $A \lessapprox B$ often really means $A \le \delta^{-\eta} B$,  where $\eta$ is a small parameter.   There are many different small parameters $\eta_1,  \eta_2, ...$ that appear in the proof,  and we need to keep track of them.   We need to choose some of these small parameters much larger than others and we need to keep track of this and make sure it is consistent.

\subsection{Keeping track of the shading}

We began with a set of tubes $\TT$ equipped with a shading $Y$.   As we study $\TT$ throughout our proof,  we met many new sets of tubes or sets of convex sets with names like $\TT_\rho$ or $\WW$.   Each of these also has a shading.   In the full proof,  we need to define all of these shadings.  For each shading we define,  we need to check that $\lambda(Y) \gtrapprox 1$. 

Defining the shading is important for checking the fundamental inequalities of the form $\mu(\TT) \lessapprox \mu(\TT[W]) \mu(\WW)$.   Defining the shading and checking this inequality requires some careful pigeonholing.  This inequality does not hold in complete generality but it does hold in all the cases where we need it.

\end{document}